\documentclass[11pt,twoside]{amsart}
\usepackage[dvips]{graphicx}
\usepackage[usenames,dvipsnames]{xcolor}
\usepackage{colortbl}
\usepackage{multirow}
\usepackage{pgf,tikz}
\usepackage{tkz-graph}
\usepackage{subfig}
\usepackage{amsmath, amsthm, amscd, amsfonts, amssymb, color}
\usepackage[bookmarksnumbered, plainpages]{hyperref}
\usepackage{relsize}
\usepackage{longtable}
\newtheorem{thm}{Theorem}[section]
\newtheorem{cor}[thm]{Corollary}
\newtheorem{lem}[thm]{Lemma}
\newtheorem{prop}[thm]{Proposition}
\newtheorem{defn}[thm]{Definition}
\newtheorem{rem}[thm]{Remark}

\newtheorem{Conj}[thm]{Conjecture}
\numberwithin{equation}{section}

\begin{document}

\oddsidemargin 0mm
\evensidemargin 0mm

\thispagestyle{plain}

\vspace{5cc}

\begin{center}

{\large\bf Cardinality of product  sets in torsion-free groups and  applications in group algebras}
\rule{0mm}{6mm}\renewcommand{\thefootnote}{}
\footnotetext{
{\rule{2.4mm}{0mm} \scriptsize Keywords and Phrases:  Torsion-free group; Unique product group; Group algebra; Sumset; Product of sets.}}

{\rule{2.4mm}{0mm} }

\vspace{1cc}
{\large\it Alireza Abdollahi and Fatemeh Jafari}

\vspace{1cc}
\parbox{24cc}{{\small 
Let $ G $ be a  unique product group, i.e.,  for any two finite subsets $A,B$ of $ G $ there exists $ x\in G $
which can be uniquely expressed as a product of an element of $A$ and an element of $B$. We prove that,  if  $C$ is a  finite  subset of  $G$ containing the identity element such that $\left\langle C\right\rangle $ is not abelian, then for all subsets $ B $ of $ G $ with $ |B|\geq 7 $, $|BC|\geq |B|+|C|+2$. 
Also, we prove that if  $C$ is a finite  subset containing the identity element of a  torsion-free group $G$ such that $ |C|=3 $ and $\left\langle C\right\rangle $ is not abelian, then for all subsets $ B $ of $ G $ with $ |B|\geq 7 $, $|BC|\geq |B|+5$. Moreover, if $\left\langle C\right\rangle$ is not isomorphic to the Klein bottle group, i.e., the group with the presentation $ \left\langle x,y\;|\; xyx=y\right\rangle  $, then   for all subsets $ B $ of $ G $ with $ |B|\geq 5 $, $|BC|\geq |B|+5$.
 The support of an  element $ \alpha= \sum_{x\in G}\alpha_xx$ in a group algebra $\mathbb{F}[G]$ ($\mathbb{F}$ is any field), denoted by $supp(\alpha)$, is the set $ \{x \in G\;|\;\alpha_x\neq 0\} $.
By the latter result, we prove that if $\alpha\beta= 0$ for some non-zero
$\alpha,\beta\in \mathbb{F}[G]$ such that $|supp(\alpha)| = 3$, then $|supp(\beta)|\geq 12$. Also, we prove that if $\alpha\beta= 1$ for some 
$\alpha,\beta\in \mathbb{ F}[G]$ such that $|supp(\alpha)| = 3$, then $|supp(\beta)|\geq 10$. These results  improve a part of results in  Schweitzer [J. Group Theory, 16 (2013), no. 5, 667-693] and   Dykema et al. [Exp. Math., 24 (2015), 326-338] to arbitrary fields, respectively. }}
\end{center}
\vspace{1cc}
% -----------------------------------------------------------
\vspace{1cc}
\begin{center}
\section {\bf Introduction and Results}
\end{center}
Let  $G$ be a torsion-free group  written multiplicatively, and let $ |.| $ denote the cardinality of a finite set. One of the basic problems in Additive Number Theory is  to obtain  lower bounds for the cardinality of $ BC=\{bc\;|\; b\in B ,\;c\in C\}$ of two finite subsets $ B $ and $ C $ in
terms of $ |B| $ and $ |C| $. There are several related results if $G$ is abelian (see  \cite{Freiman1, Freiman2, Gardner, Green, Ruzsa1,Robinson,  Ruzsa2},  for instance). But, in the non-abelian case the situation is much less understood. By the main result of \cite{set2},
\begin{equation}\label{kemperman}
|BC|\geq |B|+|C|-1.
\end{equation}
In \cite{BF},  it is proved that, for  $|B|,|C|\geq 2$, the equality in  Equation \ref{kemperman} holds if  there exist $ b\in B^{-1} $ and $ c\in C^{-1} $ such that both $bB$ and $Cc$ are progressions with common ratio that is, a set of the form $ \{a,ar,\dots,ar^{n-1}\} $, for some commuting elements $a$ and $  r$ of $G$ and some integer $n$ ($r$ will be called the ratio of the progression and $n$ its length). It is proved in \cite{Hamidoune} that if $ C $ is a finite  subset containing the identity element of a  torsion-free group $G$ such that  $\left\langle C\right\rangle $ is not abelian, then for all subsets $ B $ with  $ |B|\geq 4 $, 
\begin{equation}\label{Hamidoune}
|BC|\geq |B|+|C|+1.
\end{equation}
Also, in \cite{palfy}, it is proved that if $ B $ is not contained in the left coset of a cyclic subgroup, and $ |C|\geq 32(3+k)^6 $, for $k\geq 1 $, then 
\[
|BC|\geq |B|+|C|+k.
\]
By using the so-called isoperimetric method, see  \cite{Hamidoune2}, \cite{Hamidoune3} or
\cite{Hamidoune1}, and a  multigraph $S(B,C)  $  associated
with a pair  $ (B,C) $ of subsets in a group (see Definition \ref{defgraph}, below), we prove that the lower bound  of $ |BC| $ in \ref{Hamidoune} can be improved for unique product groups. Recall that a group $G$ has the unique product property, if for every
pair of finite non-empty sets $A,B\subseteq G$, there exists an element $x\in AB$ such that $|\{(a,b)\in A\times B \;|\; x=ab \}|=1$.
Note that every unique product group is torsion-free. Obvious examples of unique product groups are right or left linearly orderable
groups and so including torsion-free nilpotent groups. In fact the class of unique product groups are very vast as it is closed under taking extensions, subdirect products and being local graded (see \cite[pp. 111, Theorem 26.3]{PI}, \cite{Rudin} and \cite{Hig}); it was first shown in \cite{Rips}
that not all torsion-free groups have the unique product property by constructing a group with the use of 
a generalization of small cancellation theory so-called graphical small cancellation theory.  
Later in \cite{Pro} a non-unique product group as a subgroup of the direct product of 3 infinite dihedral groups was found. The notion of unique product groups was first considered to settle the existence of non-zero zero divisors
in  group algebras of such groups (see Higman's PhD Thesis \cite{Hig2}); the unit conjecture of group algebras of torsion-free groups cannot be easily settled in the class of unique product groups and to solve the problem a similar slightly stronger notation of ``two unique product group"
 has been considered (a two unique product group is defined as a unique product group, where in the definition, one replaces ``an element" by ``two distinct elements" and consider the obvious necessary condition ``$|A|>1$ or $|B|>1$"). However  these two notations are the same \cite{stro}.

Our main results are as follows:
\begin{thm}\label{uptable}
Let $G$ be a  unique product group and  $C$ be a  finite  subset of $G$ containing the identity element such that  $\left\langle C\right\rangle$ is not abelian.  Then for all subsets $ B $ of $ G $ with $ |B|\geq 7 $,
$|BC|\geq |B|+|C|+2$.
\end{thm}
In  Lemmas \ref{7part}, \ref{maintheorem} and \ref{mainresut7} (see below), we determine    the structure of  $ B $ and $ C $ with $ |C|=3 $ and  $ |B|\in \{4,5,6\} $  satisfying    the equality of  \ref{Hamidoune}. By Lemmas \ref{maintheorem} and \ref{mainresut7} (see below), if $ |B|\in \{5,6\}$, then $ G $ is isomorphic to the  Klein bottle group i.e., the group with the presentation $ \left\langle x,y\;|\; xyx=y\right\rangle  $.
\begin{thm}\label{main}
Let $G$ be a  torsion-free group and  $C$ be a  finite  subset of $G$ containing the identity element such that $|C|=3$ and  $\left\langle C\right\rangle$ is not abelian. Then for all subsets $ B $ of $ G $ with $ |B|\geq 7 $, $|BC|\geq |B|+5$. In particular, if $\left\langle C\right\rangle$ is not isomorphic to the Klein bottle group, then  for all subsets with $ |B|\geq 5 $, $ |BC|\geq |B|+5 $.
\end{thm}
Table \ref{10101} summarizes some  results for the  lower bound of $|BC|$, where $B$ and $C$ are two finite subsets of a torsion-free group.\\
We note that the support of an  element $ \alpha= \sum_{x\in G}\alpha_xx$ in a group algebra, denoted by $supp(\alpha)$, is the set $ \{x \in G\;|\;\alpha_x\neq 0\} $. The following are corollaries of  Theorem \ref{main}.
\begin{cor}\label{11}
Let $  \alpha$ and $  \beta$ be non-zero elements of the group algebra of any torsion-free group over an
arbitrary field. If $|supp(\alpha)| = 3$ and $\alpha \beta= 0$, then $|supp(\beta)| \geq 12$.
\end{cor}
\begin{cor}\label{12}
Let $ \delta$ and $  \gamma$ be  elements of the group algebra of any torsion-free group over an
arbitrary field. If $|supp(\delta)| = 3$ and $\delta\gamma= 1$, then $|supp(\gamma)| \geq 11$.
\end{cor}
Corollaries \ref{12} and \ref{11} improve the lower bounds of 
10 and 9  in    \cite[Theorem 1.7]{AT} and \cite[Theorem 1.4]{AT}, respectively.
% -----------------------------------------------------------------------------------------------------------------------------
\begin{table}
\begin{tabular}{|c|c|c|c|c|c|}
\hline
$G$&  $\left\langle C\right\rangle$  &$|C|$& $|B|$& Lower bound for $|BC|$& Reference\\ \hline
\multirow{4}{*}{Torsion-free}&Arbitrary&$\geq 2$  & $\geq 2$ & $ |B|+|C|-1$  &\cite{set2} \\ \cline{2-6}
& \multirow{2}{*}{Non-abelian}& $\geq 3$  & $\geq 4$ & $ |B|+|C|+1$  &\cite{Hamidoune} \\\cline{3-6}
& & \cellcolor{black!20!white}{$=3$}  & \cellcolor{black!20!white}{$\geq 7$} & \cellcolor{black!20!white}{$ |B|+5$}  & \cellcolor{black!20!white}{Theorem \ref{main}}\\\cline{2-6}
& \cellcolor{black!20!white}{Non-abelian, $\ncong$ Klein bottle group}&  \cellcolor{black!20!white}{$=3$}  & \cellcolor{black!20!white}{$\geq 5$} & \cellcolor{black!20!white}{$ |B|+5$}  &\cellcolor{black!20!white}{Theorem \ref{main}} \\\hline
\cellcolor{black!20!white}{Unique product} & \cellcolor{black!20!white}{Non-abelian} & \cellcolor{black!20!white}{$\geq 3$}  & \cellcolor{black!20!white}{$\geq 7$} & \cellcolor{black!20!white}{$ |B|+|C|+2$}  &\cellcolor{black!20!white}{Theorem \ref{uptable}} \\\hline
\end{tabular}
\caption{  Results  for the  lower bound of $|BC|$, where $B$ and $ C$ are two finite subsets of a group $G$ and $1\in C$. The gray rows show the result in this paper.}\label{10101}
\end{table}
\section {\bf Preliminaries}
Let $G$ be a torsion-free group  and  $C$ be a finite subset of $G$ containing the identity element. For
$X\subseteq G$, 
$
\partial_C (X):=XC\setminus X.
$
For every positive integer $k$, the $k$-isoperimetric number of $C$ is
\[
\kappa_k(C):=\text{min}\big{\{}|\partial_C (X)|\;\big{|}\;X\subseteq G, k\leq |X|<\infty \big{\}}.
\]
A finite subset $X$ of $G$ is a $k$-critical set of $C$ if $|X|\geq  k$ and $ |\partial_C (X)|= \kappa_k(C)$. A $k$-atom
of $C$ is a $k$-critical set of $C$ with minimal cardinality. The cardinality of a $k$-atom of $ C $ will be denoted by $\alpha_k(C)$.\\
The minimality of the cardinality of atoms directly yields the following lemma that we repeatedly  use in the  paper.
\begin{lem}{\rm(}\cite[Lemma 4]{Hamidoune}{\rm)}\label{lem4H}
Let $1\in C$ be a finite generating subset of a torsion-free group $G$. Let $A$
be a $k$-atom of $C$ such that $|A|>k$. Then  $|zC^{-1} \cap  A| \geq 2$  for all  $z \in AC$. Moreover, 
$|A|(|C| - 2)\geq 2\kappa_ k (C)$.
\end{lem}
\begin{defn}\label{set}
Let $B$ and $C$ be  finite non-empty subsets of $G$.  For each element $ x\in BC$, denote by  $R_{BC}(x)$ the set $\{(b,c)\in B\times C\;|\; x=bc\}$ and let $r_{BC}(x):=|R_{BC}(x)| $. Clearly,  $ r_{BC}(x)\leq min\{|B|,|C|\} $ for all  $ x\in BC $.
\end{defn} 
In \cite{Hamidoune1}, Y.O. Hamidoune proposed the following conjecture:
\begin{Conj}\label{con1}
Let $G$ be a torsion-free group,  $C$ a finite  subset of $G$ containing the identity element and $n$  a positive integer. Then $|\alpha_n(C)| = n$.
\end{Conj}

In view of \cite[p. 5, lines 27-31]{palfy}, Conjecture {\rm\ref{con1}} holds valid for unique product groups. Conjecture \ref{con1} is also valid for $ n=1 $ (see the proof of \cite[Proposition 2.8]{Hamidoune2} or \cite[Corollary 2]{Hamidoune}).  Conjecture \ref{con1} is still open and attempts to confirm it even  for $ n=2 $ were  unsuccessful see e.g. \cite{Hamidoune,palfy}.

It was known that some problems in ``product of finite subsets of groups" are related to the problem of 
``absence of zero divisors in the group ring of a torsion-free group over an
integral domain" (see \cite[p. 463, lines 5-8]{BF}). Here we  show that Conjecture \ref{con1} is related to
zero divisor and unit conjectures on group algebras of torsion-free groups. Let us first state the latter conjectures.  

In 1940, Irving Kaplansky \cite{kaplansky} stated his well known conjecture so-called zero divisor conjecture, as follows:
\begin{Conj} \label{zk} Let $  \mathbb{F}$ be any field and $G$ any torsion-free group. Then $\mathbb{F}[G]$ contains no  zero divisor, where a zero divisor is a non-zero element $\alpha\in \mathbb{F}[G]$  such that $ \alpha\beta=0 $ for some non-zero element $\beta \in \mathbb{F}[G]$. 
\end{Conj}
Another famous conjecture of Kaplansky on group algebras so-called unit conjecture, is the following \cite{kaplansky}:
\begin{Conj}\label{uk}
Let $\mathbb{F}$ be any field and $G$ any torsion-free group. Then $\mathbb{F}[G]$ contains no  non-trivial units, where  trivial units are non-zero scalar multiples of group elements.
\end{Conj}
Conjecture \ref{uk} is actually  stronger than   Conjecture \ref{zk} so that  the affirmative solution to Conjecture \ref{uk} implies the positive one for Conjecture \ref{zk} \cite[Lemma 13.1.2]{PI}.
Partial results have been obtained on  Conjectures \ref{uk} and \ref{zk}  \cite{Brown, a2, a13,  a5, a14, a6, PI, a1}.\\
Now we show connections of  Conjectures \ref{uk} and \ref{zk} to Conjecture \ref{con1} in the following two propositions.
\begin{prop}\label{conz}
Let $G$ be a torsion-free group and  $C$ be a finite  subset of $G$ containing the identity element such that $|\alpha_k(C)| > k$. If $A$ is a $k$-atom of $C$ and $|AC|= \frac{|A||C|}{2}$, then there exists a counterexample to Conjecture {\rm\ref{zk}}. 
\end{prop}
\begin{proof}
Since $|\alpha_k(C)| > k$,   $ r_{AC}(x)\geq 2 $ for all $ x\in AC $ (see Lemma \ref{lem4H}). Now it follows from $|AC|= \frac{|A||C|}{2}$  that  $r_{AC}(x)=2$ for all $x\in AC$.  We now define $ \alpha $ and $ \beta $, respectively, as the elements $ \alpha:=\sum_{a\in A}a $ and $ \beta:=\sum_{c\in C}c $ in the group algebra $\mathbb{F}_2[G]$ of $G$ over the field $\mathbb{F}_2$ with $2$
elements. Therefore   $ \alpha\beta=\sum_{x\in AC} r_{AC}(x) x=0$. This completes the proof.
\end{proof}
\begin{prop}\label{conu}
Let $G$ be a torsion-free group and  $C$ be a finite  subset of $G$ containing the identity element such that $|\alpha_k(C)| > k$. If $A$ is a $k$-atom of $C$ and $|AC|= \frac{|A||C|-1}{2}$, then there exists a counterexample to Conjecture {\rm\ref{uk}}. 
\end{prop}
\begin{proof}
Since $|\alpha_k(C)| > k$,  $r_{AC}(x)\geq 2 $ for all $ x\in AC $ (see Lemma \ref{lem4H}).   Now it follows from   $|AC|= \frac{|A||C|-1}{2}$  that  there exists $ x\in AC $ such that $ r_{AC}(x)= 3 $ and $ r_{AC}(x')= 2 $, for all $ x'\in AC\setminus \{x\}$. Let $ (a_0,c_0)\in R_{AC}(x) $ and define $ \alpha $ and $ \beta $,
respectively, as the elements
$ \alpha=\sum_{a\in A}a $ and
$ \beta=\sum_{c\in C}c $ in the group algebra $\mathbb{F}_2[G]$. Then  $$ a_0^{-1}\alpha\beta c_0^{-1}=a_0^{-1}\left(\sum_{y\in AC} r_{AC}(y) y\right) c_0^{-1}=a_0^{-1} (3 x) c_0^{-1}  =a_0^{-1} (a_0 c_0) c_0^{-1}= 1. $$
 This completes the proof.
\end{proof}
In what follows, we present two lemmas  which we use to prove Theorem  \ref{alpha}.
\begin{lem}{\rm(}\cite[Lemma 5]{Hamidoune}{\rm)}\label{lem5H}
Let $C$ be a finite generating subset of a torsion-free group $G$ such
that $|C|\geq 3$ and $1\in C$. Let $A$ be a $2$-atom of $C$. Then $|A|\leq |C|-1$.
\end{lem}
\begin{lem}{\rm(}\cite[Lemma 5]{palfy}{\rm)}\label{lem5P}
For a torsion-free group $G$ and $n \geq 2$, if $A$ is an $n$-atom for $C \subseteq G$ and
$g \in G\setminus\{ 1\}$, then
\[
|A\cap Ag|\leq \frac{n-2}{n-1}|A|+\frac{1}{n-1}\leq \frac{n}{n-1}|A|.
\]
\end{lem}
 Using Propositions \ref{conz} and \ref{conu}, we give a partial answer to Conjecture \ref{con1}.
\begin{thm}\label{alpha}
Let $G$ be a torsion-free group and  $C$ be a finite  subset of $G$ containing the identity element. If $ |C|\leq 5 $, then $ \alpha_2(C)=2 $. Moreover, if $|C|\leq 7 $, then $\alpha_2(C)\neq 3$.
\end{thm}
\begin{proof}
Observe that if $ |C|=1 $, then each subset with two elements of $ G $ is a 2-atom of $C$ and also if $ |C|=2 $, then according to \ref{kemperman}, $C$ is a 2-atom of itself. Now, suppose that $|C|\geq 3 $ and $\alpha_2(C)>2$.  Let $ A $ be a 2-atom of $ C $ containing the identity element. It follows from Lemma \ref{lem5H}  that $|A|\leq |C|-1 $. Hence, $ |C|\geq 4 $. We first show that $\kappa_2(C)\geq |C|$. In view of \ref{kemperman}, $\kappa_2(C)\geq |C|-1$. If $\kappa_2(C)= |C|-1$, then by the main result of \cite{BF}, there exist $ r\in G\setminus \{1\} $  and $ a\in G $ such that $A=\{a,ar,\ldots,ar^l\}$ for some non-negative integer $l$  and therefore $ |A\cap Ar|\geq 2 $, contradicting  Lemma \ref{lem5P} . So, $\kappa_2(C)\geq |C|$. Now since  $|A|>2$, Lemma \ref{lem4H} implies that $r_{AC}(x)\geq 2 $ for all $ x\in AC $. It follows that $ |AC|\leq \frac{|A||C|}{2} $. Thus,
\begin{equation}\label{t}
|C|+|A|\leq |AC|\leq \frac{|A||C|}{2}.
\end{equation}
If $ |C|=4 $, then $ |A|=3 $ which contradicts \ref{t}.
So far we have proved that if $|C|\leq 4$ then $\alpha_2(C)=2$; we will use the latter in the rest of the proof.\\
 Now, suppose that $ |C|=5 $. Then $ |A|\in \{3,4\} $. Now  \ref{t} implies that $ |A|\neq 3 $ and so  $ |A|=4 $. So,  \ref{t} implies $ |AC|\in\{9,10\} $. Suppose first that $ |AC|=10 $. Thus, according to the proof of Proposition \ref{conz}, there are $ \alpha,\beta \in \mathbb{F}_2[G] $ such that $ supp(\alpha)=A $, $ supp(\beta)=C $ and $ \alpha\beta=0 $, contradicting \cite[Theorem 1.3]{PS}. Now, suppose that $ |AC|=9 $. Therefore, $|C^{-1}A^{-1}|=9$ leads to $\kappa_2(A^{-1})\leq |A^{-1}|$. Since $|A^{-1}|=4$, by the above,  $ \alpha_2(A^{-1})=2 $. Let $A'=\{1,r\}$ be a 2-atom of $A^{-1}$. From $ |A'A^{-1}|\leq 6 $ and \ref{kemperman} it follows that $ |A\cap Ar|\geq 2 $ contradicting  Lemma \ref{lem5P} . This completes the proof of the first part of the theorem. If $ |C|=6 $ and  $ |A|=3 $, then by \ref{t}, $|AC|=9$. Thus, according to the proof of Proposition \ref{conz}, there are $ \alpha,\beta \in \mathbb{F}_2[G] $ such that $ supp(\alpha)=A $, $ supp(\beta)=C $ and $ \alpha\beta=0 $, contradicting \cite[Theorem 1.3]{PS}. Now, suppose that $ |C|=7 $ and $ |A|=3 $. Then, \ref{t} implies $ |AC|=10$. Hence, according to the proof of Proposition \ref{conu}, there are $ \alpha,\beta \in \mathbb{F}_2[G] $ such that $|supp(\alpha)|=|A|$, $|supp(\beta)|=|C| $ and $ \alpha\beta=1 $, contradicting \cite[Proposition 4.12]{a55}. This completes the proof.
\end{proof}
%--------------------------------------------------------------------------------------------------------------------
\vspace{1cc}
\subsection {\bf   Product set graph}
We follow the definitions and notations on graphs as in \cite{AT}. Note that by a graph we mean a triple $(\mathcal{V},\mathcal{E},\psi)$, where $\psi$ is a function which appears  if the edge set $\mathcal{E}$ is non-empty (for details see \cite[p. 2, paragraph 3]{AT}).  In \cite[Definitions 2.1 and 2.7]{AT}, two multigraphs $Z(\alpha ,\beta)$ and $U(a, b)$ are associated
with a pair $(\alpha,\beta)$ of zero divisors {\rm(}i.e. $\alpha\beta=0 ${\rm)} and a pair $(a,b)$ of unit elements {\rm(}i.e. $ab=1${\rm)} in a group algebra, respectively. The definitions of the latter graphs  are independent of any field and the conditions on $\alpha$ and $\beta$ or $a$ and $b$; in fact, the definitions only depend on the product of two pair of subsets  $(supp(\alpha),supp(\beta))$ and $(supp(a),supp(b))$. Here, we extend these definitions to the product of any pair of non-empty finite subsets of a group as follows:
\begin{defn}\label{defgraph}
Let  $G$ be a group and   $B$, $C$ be two finite non-empty  subsets of $ G $. We assign a graph $ P(B,C) $ to $(B,C)$ called the product set graph of $(B,C)$ as follows:\\ the vertex set is $B$, the edge set is 
\[
\big{\{}\{(b,b',c,c'),(b',b,c',c)\}\,|\;c,c'\in C,\;b,b'\in B,\; b\neq b',\;bc=b'c'\big{\}},
\] 
and if $ \mathcal{E}_{P(B,C)}\neq \varnothing $, the function $ \psi_{P(B,C)}:\mathcal{E}_{P(B,C)}\rightarrow {\mathcal{V}}^{2}_{_{P(B,C)}} $ is defined by
\[
\psi_{P(B,C)}(\{(b,b',c,c'),(b',b,c',c)\})=\{b,b'\},
\]
for all $ \{(b,b',c,c'),(b',b,c',c)\}\in  \mathcal{E}_{P(B,C)}$.
\end{defn}
\begin{rem}
Notice that in general the graph $P(B,C)$  may not be a simple graph. In fact $P(B,C)$ is an undirected graph with no loops but it may happen that it  has a multi-edge. For example  two multigraphs $Z(\alpha ,\beta)$ and $U(a,b)$ associated
with a pair $(\alpha,\beta)$ of zero divisors  or a pair $(a,b)$ of unit elements  in a group algebra \cite{AT,AJ} are isomorphic to the graphs $ P(supp(\beta)^{-1},supp(\alpha)^{-1}) $ and $ P(supp(b)^{-1},supp(a)^{-1}) $, respectively.
\end{rem} 
\begin{lem}\label{iso}
Let $B$, $C$ be two finite non-empty  subsets of a group $G$. Then $P(B,C)\cong P(xB,Cy)$ for all $x,y\in G$.
\end{lem}
\begin{proof}
The proof is similar to \cite[Lemma 2.4]{AT}.
\end{proof}
\begin{thm}\label{cayley}
Let  $G$ be a torsion-free group and   $B$, $C$ be two finite non-empty  subsets of $G$ such that $1\in C$ and $|C|=3$. If $ \left\langle C\right\rangle $ is not cyclic, then $P(B,C)$  is  the induced subgraph of the Cayley graph of $G$  with respect to  $S:=\{hh'^{-1}\;|\;h\neq h', h,h'\in C\}$ on the set $B$.
\end{thm}
\begin{proof}
 Suppose that there are two edges between distinct vertices $g$ and $g'$ of $P(B,C)$. Hence, there exist  $ h_1,h_2,h'_1,h'_2\in C $ such that $\{(g,g',h_1,h'_1),(g',g,h'_1,h_1)\}$ and $\{(g,g',h_2,h'_2),(g',g,h'_2,h_2)\}$ are distinct elements of $\mathcal{E}_{P(B,C)}$.  Clearly, $ h_1\neq h'_1 $, $ h_2\neq h'_2 $, $ h_1\neq h_2 $, $ h'_1\neq h'_2 $ and  $ h_{1}{h'}_{1}^{-1}=g^{-1}g'=h_{2}{h'}_{2}^{-1} $  which implies  ${h}_{2}^{-1}h_1= {h'}_{2}^{-1}h'_1$. Then,  $|\{h^{-1}h'\;|\;h\neq h', h,h'\in C\}|<6$. Now, since $1 \in C$, it follows from \cite[Lemma 3.5]{AT}  that $ \left\langle C\right\rangle  $ is an infinite cyclic group that is a contradiction.
\end{proof}
In the following, we need some definitions that are very similar to the definitions in \cite[p. 11]{AT}. For the reader's convenience, all  necessary definitions are given  below.
\begin{defn}
Let $G$ be a torsion-free group and $B$ and  $C$ be two non-empty finite  subsets of $G$ such that $1\in C$, $|C|=3$ and  $ \left\langle C\right\rangle $ is not cyclic. Suppose that $\mathcal{C} $ is a cycle of length $ n $ in $ P(B,C) $ with the vertex set $ \{g_1,g_2,\ldots,g_n\}\subseteq B $ such that $ g_i\sim g_{i+1} $ for all $ i\in \{1,\ldots,n-1\} $ and $ g_1\sim g_n $. By an arrangement $ l $ of the vertex set  $\mathcal{C} $, we mean a sequence of all vertices as $x_1,\ldots,x_k$ such that $ x_i\sim x_{i+1} $ for all $ i\in \{1,\ldots,n-1\} $ and $ x_1\sim x_n $.
There exist unique  distinct elements $ h_i,h'_i\in C$, $i\in \{1,\ldots,n\}$, satisfying the following relations:
\begin{equation}\label{cycle}
R: \left\{
\begin{array}{ll}
 g_{1} h_{1}=g_{2}h'_{1} \\  
g_{2}h_{2}=g_{3}h'_{2} \\
\vdots\\
g_{n}h_{n}=g_{1}h'_{n}.
\end{array} \right. 
\end{equation} 
 We assign the $2n$-tuple $ T_{\mathcal{C} }^l=[h_1,h'_1,h_2,h'_2,\ldots ,h_n,h'_n] $  to   $\mathcal{C} $ corresponding to the above arrangement $ l $ of the vertex set of  $\mathcal{C}$. We denote by $ R( T_{\mathcal{C}}^l) $ the above set $R$ of relations.
It can be derived from the relations {\rm\ref{cycle}} that $ r( T_{\mathcal{C}}^l):=(h_{1}{h'}_{1}^{-1})({h}_{2}{h'}_{2}^{-1})\ldots (h_{n}{h'}_{n}^{-1}) $ is equal to $1$. It
follows from Lemma {\rm\ref{cayley}}  that if $ T_{\mathcal{C}}^{l'} $ is  the $ 2n $-tuple of $ \mathcal{C} $ corresponding to another arrangement $ l' $ of the vertex set of $\mathcal{C}$, then $ T_{\mathcal{C}}^{l'} $ is one of the following $ 2n $-tuples:
 \begin{center}
  \begin{tabular}{c}
  $ [h_1,h'_1,h_2,h'_2,\ldots ,h_n,h'_n] $,\\
 $ [h_2,h'_2,h_3,h'_3,\ldots ,h_n,h'_n,h_1,h'_1] $,\\
 $ \vdots $\\
 $ [h_n,h'_n,h_1,h'_1,\ldots ,h_{n-2},h'_{n-2},h_{n-1},h'_{n-1}] $,\\
$ [h'_n,h_n,h'_{n-1},h_{n-1},\ldots ,h'_2,h_2,h'_1,h_1] $,\\
$ \vdots $\\
$ [h'_1,h_1,h'_{n},h_{n},\ldots ,h'_3,h_3,h'_2,h_2] $.
 \end{tabular}
 \end{center}
 The set of all such $ 2n $-tuples will be denoted by $ \mathcal{T}(\mathcal{C})$.  
 \end{defn}
\begin{defn}
Let $G$ be a torsion-free group and $B$ and  $C$ be two finite non-empty subsets of $G$ such that $1\in C$, $|C|=3$ and  $ \left\langle C\right\rangle $ is not cyclic. Let $\mathcal{C}$ be a cycle of length $ n $ in $ P(B,C) $. Since $ r(T_1)=1 $ if and only if $ r(T_2)=1 $, for all $ T_1,T_2\in  \mathcal{T}(\mathcal{C}) $, a member of $ \{r(T)\;|\;T \in  \mathcal{T}(\mathcal{C}) \} $ is given as a representative and denoted by $ r(\mathcal{C}) $. Also, $ r(\mathcal{C})=1 $  is called the relation of $\mathcal{C}$.
\end{defn}
 \begin{defn}\label{eqquidef}
Let $G$ be a torsion-free group and $B$ and  $C$ be two finite non-empty subsets of $G$ such that $1\in C$, $|C|=3$ and  $ \left\langle C\right\rangle $ is not cyclic. Let $\mathcal{C}$ and $\mathcal{C'}$ be two cycles of length $ n $ in $  P(B,C) $. We say that these two
cycles are non-equivalent, if $ \mathcal{T}(\mathcal{C})\cap \mathcal{T}(\mathcal{C'})=\varnothing $.
\end{defn}
 \begin{rem}\label{eqquirem}
 Suppose  that $\mathcal{C}$ and $\mathcal{C'}$ are two cycles of length $ n $ in $  P(B,C)  $. Then  $\mathcal{C}$ and $\mathcal{C'}$ are equivalent if $ \mathcal{T}(\mathcal{C})=\mathcal{T}(\mathcal{C'})$.
 \end{rem}
 \begin{rem}\label{type}
Let $G$ be a torsion-free group and $B$ and  $C$ be two finite non-empty   subsets of $G$ such that $1\in C$, $|C|=3$ and  $ \left\langle C\right\rangle $ is not cyclic. Suppose that $\mathcal{C}$ is a cycle of length $3$ {\rm(}a triangle{\rm)} in $P(B,C) $ with the vertex set  $\mathcal{V}_{\mathcal{C}}=\{g_1,g_2,g_3\}\subseteq B $.  Suppose further that $ T\in \mathcal{T}(\mathcal{C}) $ and $ T=[h_1,h'_1,h_2,h'_2,h_3,h'_3] $, where $ h_i,h'_i\in C $ for all $ i\in \{1,2,3\} $. Then $ T $ satisfies exactly one of the following conditions:
\begin{itemize}
\item[$(i)$] $ h_1\neq h'_1=h_2\neq h'_2=h_3\neq h'_3=h_1. $
\item[$(ii)$]$ h_1\neq h'_1\neq h_2\neq h'_2\neq h_3\neq h'_3\neq h_1. $
\end{itemize}
More precisely, we shall speak of a triangle $\mathcal{C}$ in $ P(B,C) $ of type $ (j) $ if $ T\in \mathcal{T}(\mathcal{C}) $ satisfies the condition $ (j) $ in the above list {\rm(}$ j $ being $ i $ or $ ii ${\rm)}.
 \end{rem}
 \begin{lem}\label{triii}
Let $G$ be a torsion-free group and $B$ and  $C=\{1,x,y\}$ be two finite non-empty  subsets of $G$ such that  $ \left\langle C\right\rangle $ is not abelian. If $ P(B,C) $ contains a triangle of type $ (ii) $, then $ x,y $ satisfy exactly one of the relations as follows: $(1)\; yxy=x $; $ (2)\; xyx=y$; $ (3)\; x^2=y^2 $. Moreover, every two triangles  of $ P(B,C) $ are equivalent.
 \end{lem}
 \begin{proof}
 Suppose that $\mathcal{C}$ is a triangle of type $ (ii) $ in $P(B,C) $. It can be seen that there are $ 13 $ non-equivalent cases for $\mathcal{C}$ and $ r(\mathcal{C}) $ corresponding to such cases are  the elements of the set 
 \begin{align*}
 A=&\big\{ x^{-3},(y^{-1}x)^2y^{-1},x^{-2}y^{-1},x^{-1}y^{-2},x^{-1}y^{-1}xy^{-1},x^{-1}yxy^{-1},
  x^{-1}yx^{-1}y^{-1},\\&(x^{-1}y)^2x^{-1},y^{-3},y^{-2}xy^{-1},x^{-1}y^{2}x^{-1},(x^{-1}y)^3,x^{-2}yx^{-1}\big\}.
 \end{align*}
 If $ r(\mathcal{C}) $ is  $i$-th element of the set $ A $, where   $ i\in\{2,3,4,6,8,10,13\} $,  then $ r(\mathcal{C})=1 $ implies that  $ \left\langle C\right\rangle $ is  abelian  that is a contradiction. If $ r(\mathcal{C}) $ is 1-th or 9-th element of the set $ A $, then $ G $ has a non-trivial torsion  element, a contradiction. Hence, $ r(\mathcal{C}) $ is  one of the elements 5-th, 7-th  or 11-th of the set $ A $. On the other hand, in a torsion-free group, at most one of these three relations can hold.  Indeed,  $ yxy=x $ and $  xyx=y$ imply $ x=xyx^2y $ and therefore $ x^2=y^{-2} $. So, $ y^{-2}=x^2=yxy^2xy=y^2 $, a contradiction. If  $ yxy=x $ and $  x^2=y^2$, then $ y^2=y^6 $, a contradiction. Hence, two triangles of type $ (ii) $ are equivalent. Observe that two triangles of type $ (i) $ are equivalent. This completes the proof.
 \end{proof}
 \begin{lem}\label{two cycle with one edge}
 Let $G$ be a torsion-free group and $B$ and  $C$ be two non-empty finite  subsets of $G$ such that $1\in C$, $|C|=3$ and  $ \left\langle C\right\rangle $ is not abelian. Suppose that $ P(B,C) $ contains a subgraph isomorphic to the graph in Figure {\rm\ref{two cycle with one edge in common}}. Then there exist $ h\in C $ and $ a'\in G $ such that  $ Ch^{-1}=\{1,a,b\} $,  $ a^2=b^2 $ and one of the sets $a'\{1,a^{-1},b^{-1},b^{-1}a\}$ and $a'\{1,b^{-1},b,ba^{-1}\}$ is a subset of $ B $.
  \end{lem}
 \begin{proof}
Let  $C=\{1,x,y\}$. Suppose that  $\mathcal{C}$ and $\mathcal{C'}$ are   two triangles with the vertex sets $ \{g_1,g_2,g_3\} $ and $ \{g_1,g_2,g_4\} $  in $ P(B,C) $ as Figure \ref{two cycle with one edge in common}. Suppose further that $ T=[h_1,h'_1,h_2,h'_2,h_3,h'_3] $ is the 6-tuple to $\mathcal{C}$ corresponding to the arrangement $ g_1,\,g_2,\,g_3 $ and also $T'=[h_1,h'_1,t_2,t'_2,t_3,t'_3] $  is the 6-tuple to $\mathcal{C'}$ corresponding to the arrangement $ g_1,\,g_2,\,g_4 $, where the first two components of $ T $ and $ T' $ are related to the common edge between these triangles.  It is clear that  either  $ h_2\neq t_2 $ or $ h_2=t_2 $ and $ h'_2\neq t'_2 $ and also either $ h'_3\neq t'_3 $ or $ h'_3=t'_3 $ and $ h_3\neq t_3 $ since otherwise   $ g_3= g_4 $.  Hence, $ T\neq T' $ and therefore it is impossible that the triangles $\mathcal{C}$ and $\mathcal{C'}$ are both  of type $ (i) $.  So, without loss of generality, we may assume that $\mathcal{C}$ is a triangle of type $ (ii) $. Then Lemma \ref{triii} implies that    $ r(\mathcal{C})\in \{ x^{-1}y^{-1}xy^{-1},  x^{-1}yx^{-1}y^{-1},x^{-1}y^{2}x^{-1} \} $. If   $ r(\mathcal{C})=x^{-1}y^{-1}xy^{-1} $, then 
 \[
 \mathcal{T}(\mathcal{C})=\big\{[1,x,1,y,x,y],[1,y,x,y,1,x],[x,y,1,x,1,y],[y,x,y,1,x,1],[y,1,x,1,y,x],[x,1,y,x,y,1]\big\},
 \]
  and if $ r(\mathcal{C})=x^{-1}y^{2}x^{-1}$, then 
 \[
 \mathcal{T}(\mathcal{C})=\big\{[1,x,y,1,y,x],[y,1,y,x,1,x],[y,x,1,x,y,1],[x,y,1,y,x,1],[1,y,x,1,x,y],[x,1,x,y,1,y]\big\}.
 \] 
Now, if $\mathcal{C'}$ is a triangle of type $ (ii)  $, then by Lemma  \ref{triii}, $\mathcal{C}$ and $\mathcal{C'}$ are equivalent and therefore $ \mathcal{T}(\mathcal{C})=\mathcal{T}(\mathcal{C'}) $. But in this case, by every choice of $ r(\mathcal{C}) $ as above, there are no two   elements of $ \mathcal{T}(\mathcal{C}) $ satisfying the mentioned conditions for $ T $ and $ T' $. Thus, $\mathcal{C'}$ must be  a triangle of type $ (i)$. Hence, $ g_1h_1=g_2h'_1=g_4t_3$, where $ \{h_1,h'_1,t_3\}=C $. Note that by Lemma  \ref{iso}, we may assume that  $ g_1=1 $. Suppose first that $ r(\mathcal{C})= x^{-1}y^{-1}xy^{-1} $. Then if we let $ a=yx^{-1} $ and  $ b=x^{-1} $, then $ Cx^{-1}=\{1,a,b\} $, $ a^2=b^2 $ and if $ T $ is the 1-th and 6-th elements of $  \mathcal{T}(\mathcal{C}) $, then $ h=y=ab^{-1}=a^{-1}b $ and $\{g_1,g_2,g_3,g_4\}$ is equal to  $b\{1,a^{-1},b^{-1},b^{-1}a\} $ and $ \{1,a^{-1},b^{-1},b^{-1}a\} $, respectively; if $ T $ is the 2-th  and 5-th elements of $  \mathcal{T}(\mathcal{C}) $, then $ h=x=b^{-1} $ and $\{g_1,g_2,g_3,g_4\}$ is equal to  $\{1,b^{-1},b,b^{-1}a\}$ and $\{1,a^{-1},a,a^{-1}b\}$; if $ T $ is the 3-th and 4-th elements of $  \mathcal{T}(\mathcal{C}) $, then $ h=1 $ and $\{g_1,g_2,g_3,g_4\}$ is equal to $\{1,a^{-1},b^{-1},a^{-1}b\} $ and $a\{1,a^{-1},b^{-1},a^{-1}b\} $, respectively. Now, suppose that $ r(\mathcal{C})=x^{-1}y^{2}x^{-1} $. Then if $ T $ is the 1-th and 6-th elements of $  \mathcal{T}(\mathcal{C}) $, then $ h=y $ and $\{g_1,g_2,g_3,g_4\}$ is equal to  $\{1,x^{-1},y^{-1},x^{-1}y\} $ and $x\{1,x^{-1},y^{-1},x^{-1}y\} $, respectively; if $ T $ is the 2-th  and 5-th elements of $  \mathcal{T}(\mathcal{C}) $, then $ h=x $ and $\{g_1,g_2,g_3,g_4\}$ is equal to  $ y\{1,x^{-1},y^{-1},y^{-1}x\} $ and $\{1,x^{-1},y^{-1},y^{-1}x\} $; if $ T $ is the 3-th and 4-th elements of $  \mathcal{T}(\mathcal{C}) $, then $ h=1 $ and $\{g_1,g_2,g_3,g_4\}$ is equal to $\{1,y^{-1},y,y^{-1}x\} $ and $\{1,x^{-1},x,x^{-1}y\} $, respectively.
 Note that if $ r(\mathcal{C})=  x^{-1}yx^{-1}y^{-1} $, then by interchanging $ x $ and $ y $ in the  case $ r(\mathcal{C})= x^{-1}y^{-1}xy^{-1} $ and with the same discussion, the statement is true. This completes the proof.
 \end{proof}
  \begin{lem}\label{complete}
   Let $G$ be a torsion-free group and $B$ and  $C$ be two non-empty finite  subsets of $G$ such that $1\in C$, $|C|=3$ and  $ \left\langle C\right\rangle $ is not abelian. Then $ P(B,C) $ contains no subgraph isomorphic to the complete graph with $4$ vertices.
  \end{lem}
  \begin{proof}
  In view of the proof of Lemma \ref{two cycle with one edge}, if $ P(B,C) $ contains two triangles which have an edge in common, then  one of them should be of  type $ (i) $ and the other of type $ (ii) $. But, if $ P(B,C) $ contains a subgraph isomorphic to the complete graph with $4$ vertices, then clearly $ P(B,C) $ contains two triangles which have an edge in common and both of them are of type $ (i) $ or $ (ii) $ that is a contradiction. This completes the proof.
  \end{proof}
 \begin{figure}
\begin{tikzpicture}[scale=.9]
\draw [fill] (0,0) circle
[radius=0.1] node  [left]  {$ g_4 $};
\draw [fill] (2,0) circle
[radius=0.1] node  [right]  {$ g_3 $};
\draw [fill] (1,1) circle
[radius=0.1] node  [left]  {$ g_1 $};
\draw [fill] (1,-1) circle
[radius=0.1] node  [left]  {$ g_2 $};
\draw  (0,0) -- (.95,.95);
\draw (1.05,.95) -- (2,0) -- (1,-1)-- (0,0);
\draw  (1,.9) --  (1,-1) ;
\end{tikzpicture}
\caption{Two triangles with one edge in common.} \label{two cycle with one edge in common}
\end{figure}
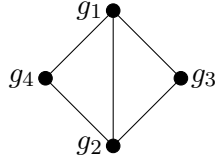
 \begin{rem}\label{squ1}
 Let $G$ be a torsion-free group and $B$ and  $C$ be two non-empty finite  subsets of $G$ such that $1\in C$, $|C|=3$ and  $ \left\langle C\right\rangle $ is not abelian. Suppose that  $\mathcal{C}$ is a cycle of length $4$ {\rm(}a square{\rm)} in $ P(B,C) $.  Suppose further that $ T\in \mathcal{T}(\mathcal{C}) $ and $ T=[h_1,h'_1,h_2,h'_2,h_3,h'_3,h_4,h'_4] $, where $ h_i,h'_i\in C $ for all $ i\in \{1,2,3,4\} $, then   by Lemma {\rm\ref{iso}} and Remark {\rm\ref{set}}, $ T $ satisfies exactly one of the following conditions:
 \begin{itemize}
 \item[(i)]$  h_{i_1}\neq h'_{i_1}= h_{i_2}\neq h'_{i_2}\neq h_{i_3}\neq h'_{i_3}\neq h_{i_4}\neq h'_{i_4}\neq h_{i_1}$, where \[ (i_1,i_2,i_3,i_4)\in\big\{(1,2,3,4),(2,3,4,1),(3,4,1,2),(4,1,2,3)\big\}. \]
 \item[(ii)] $  h_1\neq h'_1\neq h_2\neq h'_2\neq h_3\neq h'_3\neq h_4\neq h'_4\neq h_1$.
\end{itemize}  
 More precisely, we shall speak of a square $\mathcal{C}$ in $ P(B,C) $ of type $ (j) $ if  $ T\in \mathcal{T}(\mathcal{C}) $  satisfies  the condition $ (j) $ in the above list {\rm(}$ j $ being $ i $ or $ ii ${\rm)}. Note that if $\mathcal{C}$ is a square of type $ (i) $, then $\mathcal{C}$ is a square obtained from two triangles with one edge in common as Figure {\rm\ref{two cycle with one edge in common}}. 
 \end{rem}
 \begin{rem}\label{squ}
Let $G$ be a torsion-free group and $B$ and  $C=\{1,x,y\}$ be two non-empty finite  subsets of $G$ such that $\left\langle C\right\rangle $ is not abelian. Suppose that  $\mathcal{C}$ is a square of type $ (ii) $ in $ P(B,C) $.
By using GAP \cite{gap}, we know that  there are $36$ non-equivalent cases for $\mathcal{C}$. The relations of such non-equivalent cases are listed in the column
labelled by R of Table {\rm\ref{C4}}. It is easy to see that each of the cases $ 2,\;3,\;10,\;12,\;24,\;30,\;32,\;34,\;35$ and $36 $ leads to being $\left\langle C\right\rangle $ is  an abelian group;   each of the cases $ 1$ and $31 $ implies $ G $ has a non-trivial torsion element. Thus, $ r(\mathcal{C})=1 $ satisfies one of the relations marked by {\rm ``*"} s in
the column labelled by $E$ of Table {\rm\ref{C4}}. We note that each relation which leads to being $\left\langle C\right\rangle $ an abelian group or $G$ having a non-trivial torsion element is marked
by an A or a T in the column labelled by E, respectively.
 \end{rem}
 \begin{rem}\label{klein}
The Klein bottle group has the  presentation   $ \left\langle x,y \mid xyx=y \right\rangle  $, also  $ \left\langle x,y \mid x^{2}=y^{2} \right\rangle$ is another presentation of the Klein bottle group. We note that by the first part of the proof of \cite[Theorem 3.1]{PS},
every torsion-free quotient of the Klein bottle group is either abelian or it is isomorphic to the Klein bottle group itself. 
\end{rem} 
 \begin{longtable}{|l|l|l|l|l|l|}
\caption{Non-equivalent relations of a cycle of length $4$  in $ P(B,C) $.}\label{C4}\\
\hline
\endfirsthead
\multicolumn{4}{c}%
{{\bfseries \tablename\ \thetable{} -- continued from previous page}} \\\hline
\endhead
\hline \multicolumn{4}{l}{{Continued on next page}} \\
\endfoot

\endlastfoot
n&R&E&n&R&E\\
\hline 
$1$&$x^{-4}=1$& T&$19$&$x^{-1}y(xy^{-1})^2=1$&$ * $\\
$2$&$x^{-3}y^{-1}=1$&A&$20$&$x^{-1}y^2xy^{-1}=1$&$ * $\\
$3$&$x^{-3}yx^{-1}=1$&A&$21$&$x^{-1}y^3x^{-1}=1$&$ * $\\
$4$&$x^{-2}y^{-2}=1$&$ * $&$22$&$x^{-1}y^2x^{-1}y^{-1}=1$&$ * $\\
$5$&$x^{-2}y^{-1}xy^{-1}=1$&$ * $&$23$&$x^{-1}y(yx^{-1})^2=1$&$ * $\\
$6$&$x^{-2}yxy^{-1}=1$&$ * $&$24$&$(x^{-1}yx^{-1})^2=1$&A\\
$7$&$x^{-2}y^2x^{-1}=1$&$ * $&$25$&$x^{-1}yx^{-1}y^{-2}=1$&$ * $\\
$8$&$x^{-2}yx^{-1}y^{-1}=1$&$ * $&$26$&$x^{-1}yx^{-1}y^{-1}xy^{-1}=1$&$ * $\\
$9$&$x^{-1}(x^{-1}y)^2x^{-1}=1$&$ * $&$27$&$(x^{-1}y)^2xy^{-1}=1$&$ * $\\
$10$&$(x^{-1}y^{-1})^2=1$&A&$28$&$(x^{-1}y)^2yx^{-1}=1$&$ * $\\
$11$&$x^{-1}y^{-1}x^{-1}yx^{-1}=1$&$ * $&$29$&$(x^{-1}y)^2x^{-1}y^{-1}=1$&$ * $\\
$12$&$x^{-1}y^{-3}=1$&A&$30$&$(x^{-1}y)^3x^{-1}=1$&A\\
$13$&$x^{-1}y^{-2}xy^{-1}=1$&$ * $&$31$&$y^{-4}=1$&T\\
$14$&$x^{-1}y^{-1}x^2y^{-1}=1$&$ * $&$32$&$y^{-3}xy^{-1}=1$&A\\
$15$&$x^{-1}y^{-1}xyx^{-1}=1$&$ * $&$33$&$y^{-1}(y^{-1}x)^2y^{-1}=1$&$  *$\\
$16$&$x^{-1}y^{-1}xy^{-2}=1$&$ * $&$34$&$(y^{-1}xy^{-1})^2=1$&A\\
$17$&$x^{-1}(y^{-1}x)^2y^{-1}=1$&$ * $&$35$&$(y^{-1}x)^3y^{-1}=1$&A\\
$18$&$x^{-1}yxy^{-2}=1$&$ * $&$36$&$(xy^{-1})^4=1$&A\\
\hline
%\caption{Non-equivalent relations of a cycle of length $3$ on vertices of degrees $ 4 $ in $ Z(\alpha,\beta) $.}
\end{longtable}
  \begin{lem}\label{equ}
 Let $G$ be a torsion-free group and $B$,  $C$ be two non-empty finite  subsets of $G$ such that $1\in C$. If $\left\langle C\right\rangle $ is neither abelian nor isomorphic to the Klein bottle group, then every two squares of $ P(B,C) $ are equivalent.
  \end{lem}
  \begin{proof}
Let   $C=\{1,x,y\}$. Suppose that there are two squares in $ P(B,C) $ . Since $\left\langle C\right\rangle $  is not isomorphic to the Klein bottle group, Lemma \ref{two cycle with one edge} and Remarks \ref{squ1} and \ref{klein} imply that both squares are of type $(ii)$.
Hence, it follows from  Remark \ref{squ} that the relation corresponding to each of these two squares must be one of the 24 non-equivalent cases with the relations marked by  ``*"s unless the cases 4, 9 and 33 in Table \ref{C4} (note that by Remark \ref{klein}, each of the cases 4, 9 and 33 leads to $\left\langle C\right\rangle $ is isomorphic to the Klein bottle group). When choosing two relations different  from each other, there are $ 210 $ cases. Using GAP \cite{gap}, we see in 201 cases among these 210 cases, a free group
with generators $x,\;y  $ and relations of each of such cases is finite or abelian that is a contradiction. If the ordered pair of integers $ (i,j) $ shows the numbers of two relations corresponding to each of 210 cases  in Table \ref{C4}, then for the other 9 cases \[ (i,j)\in \{(6,11),(6,15),(8,15),(13,18),(16,20),(18,20),(19,27),(19,28),(23,27)\}. \] It is easy to see that if $ (i,j) $ is 1-th or 3-th element of the set, then $ x^5=1 $;  if $ (i,j) $ is 2-th element of the set, then $ x^3=1 $; if $ (i,j) $ is 4-th or 5-th element of the set, then $ y^5=1 $;  if $ (i,j) $ is 6-th  element of the set, then $ y^3=1 $; if $ (i,j) $ is 7-th  element of the set, then $ (x^{-1}y)^3=1 $; if $ (i,j) $ is 8-th  element of the set, then $ (yx^{-1})^3=1 $; if $ (i,j) $ is 9-th  element of the set, then $ (xy^{-1})^3=1 $. Thus, each of these 9 cases  leads to $ G $ having a non-trivial torsion element, a contradiction. This completes the proof.  
  \end{proof}
 \begin{lem}\label{forbidden}
 Let $G$ be a torsion-free group and $B$,  $C$ be two non-empty finite  subsets of $G$ such that $1\in C$. If $\left\langle C\right\rangle $ is neither abelian nor isomorphic to the Klein bottle group, then $ P(B,C) $ contains no subgraph isomorphic to one of the graphs in Figure {\rm \ref{forbiddensub}}.
  \end{lem}
\begin{proof}
Let $C=\{1,x,y\}$. We note first that since $\left\langle C\right\rangle $ is not isomorphic to the Klein bottle group, Lemmas \ref{triii} and \ref{two cycle with one edge} and Remark \ref{klein} imply that all triangles and squares in $ P(B,C) $ must be of type $ (i) $ and $ (ii) $, respectively. Suppose that $ P(B,C) $ contains the graph $ \Gamma_1 $ in Figure {\rm \ref{forbiddensub}}. Suppose that $\mathcal{C}$ and $\mathcal{C'}$ are two cycles of  length $4$  in $\Gamma_1$  which have exactly two consecutive edges in common.  Further suppose that $ T\in \mathcal{T}(\mathcal{C}) $, $ T'\in \mathcal{T}(\mathcal{C'}) $, $ T=[h_1,h'_1,h_2,h'_2,h_3,h'_3,h_4,h'_4] $ and   $T'=[h_1,h'_1,h_2,h'_2,t_3,t'_3,t_4,t'_4] $, where the first four components are related to the common edges between these cycles. Lemma \ref{equ} and Remark \ref{eqquirem} imply $ \mathcal{T}(\mathcal{C})=\mathcal{T}(\mathcal{C'}) $ and therefore $ r(\mathcal{C})=r(\mathcal{C'}) $.  Also, since $\left\langle C\right\rangle $ is not isomorphic to the Klein bottle group, Lemma \ref{two cycle with one edge} implies  $ h_3\neq t_3 $ and $ h'_4\neq t'_4 $ and hence  $ T\neq T' $. Also, by Lemma \ref{iso}, we may assume that $ h_1=1 $. Clearly,  $ r(\mathcal{C})=r(\mathcal{C'})=1 $ is one of the relations marked by ``*"s  in Table \ref{C4} unless the cases 4, 9 and 33 (21 cases).  With  every choice of $ r(\mathcal{C})=1 $ of such cases, it is easy to see  that there are no two   elements of $ \mathcal{T}(\mathcal{C}) $  that  satisfy the mentioned conditions for $ T $ and $ T' $. So, $ P(B,C) $ contains no subgraph isomorphic to the graph $ \Gamma_1 $. Now suppose that  $ P(B,C) $ contains two cycles $ \mathcal{C} ,\; \mathcal{C'} $ of  length $4$  which have exactly one  edge in common. Further suppose that $ T\in \mathcal{T}(\mathcal{C}) $, $ T'\in \mathcal{T}(\mathcal{C'}) $, $ T=[h_1,h'_1,h_2,h'_2,h_3,h'_3,h_4,h'_4] $ and   $T'=[h_1,h'_1,t_2,t'_2,t_3,t'_3,t_4,t'_4] $, where the first two components are related to the common edge between these cycles. Lemma \ref{equ} and Remark \ref{eqquirem} imply $ \mathcal{T}(\mathcal{C})=\mathcal{T}(\mathcal{C'}) $ and therefore $ r(\mathcal{C})=r(\mathcal{C'})=1 $. It is easy to see that  either $ h_2\neq t_2 $ or $ h_2= t_2 $ and  $ h'_2\neq t'_2 $ and also either $ h'_4\neq t'_4 $ or $ h'_4= t'_4 $ and  $ h_4\neq t_4 $ and so $ T\neq T' $.  On the other hand, $ r(\mathcal{C})=r(\mathcal{C'})=1 $ is one of the relations marked by ``*"s  in Table \ref{C4} unless the cases 4, 9 and 33 (21 cases). It can be seen that just in the cases 5,6,7,8,11,13,14,15,16,18,19,20,21 and 22  there are two  elements of $ \mathcal{T}(\mathcal{C}) $ that  satisfy the mentioned conditions for $ T $ and $ T' $. So, there are 14 cases for the existence of   two squares  which have exactly one  edge in common in $ P(B,C) $. The graphs $ \Gamma_2 $, $ \Gamma_3 $ and $ \Gamma_4 $ obtained from adding a square to   two squares   which have exactly one  edge in common and also the graph $ \Gamma_5 $  obtains from adding  two triangles to the  two squares  which have exactly one  edge in common. By a same argument as above it can be seen that in each of  14 cases it is impossible that adding a square with properties of graphs  $ \Gamma_2 $, $ \Gamma_3 $ and $ \Gamma_4 $ or adding two triangles of type $ (i) $ with properties of the graph  $ \Gamma_5 $ to the  two squares  which have exactly one  edge in common. So, $ P(B,C) $ contains no subgraph isomorphic to one of the graphs $ \Gamma_2 $, $ \Gamma_3 $, $ \Gamma_4 $ and $ \Gamma_5 $ in Figure {\rm \ref{forbiddensub}}.
\end{proof}  
  \begin{figure}
  \subfloat{$  \Gamma_1 $}
\begin{tikzpicture}[scale=.85]
\draw [fill] (0,0) circle
[radius=0.1] node  [left]  {};
\draw [fill] (1,0) circle
[radius=0.1] node  [left]  {};
\draw [fill] (2,0) circle
[radius=0.1] node  [left]  {};
\draw [fill] (1,1) circle
[radius=0.1] node  [left]  {};
\draw [fill] (1,-1) circle
[radius=0.1] node  [left]  {};
\draw  (1,1) -- (1,0) -- (1,-1) -- (0,0)--(1,1);
\draw  (1,1) -- (2,0) -- (1,-1) ;
\end{tikzpicture}
\subfloat{$  \Gamma_2 $}
\begin{tikzpicture}[scale=.85]
\draw [fill] (0,0) circle
[radius=0.1] node  [left]  {};
\draw [fill] (1,0) circle
[radius=0.1] node  [right]  {};
\draw [fill] (2,0) circle
[radius=0.1] node  [below]  {};
\draw [fill] (0,-1) circle
[radius=0.1] node  [below]  {};
\draw [fill] (1,-1) circle
[radius=0.1] node  [below]  {};
\draw [fill] (2,-1) circle
[radius=0.1] node  [below]  {};
\draw [fill] (1,1) circle
[radius=0.1] node  [below]  {};
\draw (0,0) -- (1,0) -- (2,0) -- (2,-1)-- (1,-1)-- (0,-1)-- (0,0);
\draw  (1,0) -- (1,-1);
%\draw  (1,0) -- (2,-1) ;
\draw  (0,0) -- (1,1)-- (2,0);
\end{tikzpicture}
\subfloat{$  \Gamma_3 $}
\begin{tikzpicture}[scale=.85]
\draw [fill] (0,0) circle
[radius=0.1] node  [left]  {};
\draw [fill] (1,0) circle
[radius=0.1] node  [below]  {};
\draw [fill] (0,-1) circle
[radius=0.1] node  [below]  {};
\draw [fill] (-.5,-.5) circle
[radius=0.1] node  [below]  {};
\draw [fill] (-.5,.5) circle
[radius=0.1] node  [below]  {};
\draw [fill] (1,-1) circle
[radius=0.1] node  [below]  {};
\draw [fill] (-1,0) circle
[radius=0.1] node  [below]  {};
\draw [fill] (-1,-1) circle
[radius=0.1] node  [below]  {};
\draw (-1,0) -- (0,0) -- (1,0) ;
\draw  (1,0) -- (1,-1) -- (0,-1) -- (0,0);
\draw  (0,0) -- (-1,0)-- (-1,-1);
\draw  (-1,-1)-- (0,-1);
\draw  (0,-1) -- (-.5,-.5) -- (-.5,.5) -- (0,0);
\end{tikzpicture}
\subfloat{$  \Gamma_4 $}  
\begin{tikzpicture}[scale=.85]
\draw [fill] (0,0) circle
[radius=0.1] node  [left]  {};
\draw [fill] (1,0) circle
[radius=0.1] node  [right]  {};
\draw [fill] (2,0) circle
[radius=0.1] node  [below]  {};
\draw [fill] (0,-1) circle
[radius=0.1] node  [below]  {};
\draw [fill] (1,-1) circle
[radius=0.1] node  [below]  {};
\draw [fill] (2,-1) circle
[radius=0.1] node  [below]  {};
\draw (0,0) -- (1,0) -- (2,0) -- (2,-1)-- (1,-1)-- (0,-1)-- (0,0);
\draw  (1,0) -- (1,-1);
\draw (0,-1) .. controls (-1,.95) and (1.75,.5)  .. (2,0);
\end{tikzpicture}
\subfloat{$  \Gamma_5 $} 
\begin{tikzpicture}[scale=.85]
\draw [fill] (0,0) circle
[radius=0.1] node  [left]  {};
\draw [fill] (-.5,-1) circle
[radius=0.1] node  [right]  {};
\draw [fill] (.5,-1) circle
[radius=0.1] node  [below]  {};
\draw [fill] (2,-1) circle
[radius=0.1] node  [below]  {};
\draw [fill] (3,-1) circle
[radius=0.1] node  [below]  {};
\draw [fill] (2.5,0) circle
[radius=0.1] node  [below]  {};
\draw (0,0) -- (-.5,-1) -- (.5,-1)-- (0,0);
\draw (2.5,0) -- (3,-1) -- (2,-1)-- (2.5,0);
\draw  (0,0) -- (2.5,0);
\draw  (.5,-1) -- (2,-1);
\draw (-.5,-1) .. controls (0,-1.5) and (2.5,-1.5)  .. (3,-1);
\end{tikzpicture}
\caption{Some forbidden subgraphs of $P(B,C) $.}\label{forbiddensub}
\end{figure}
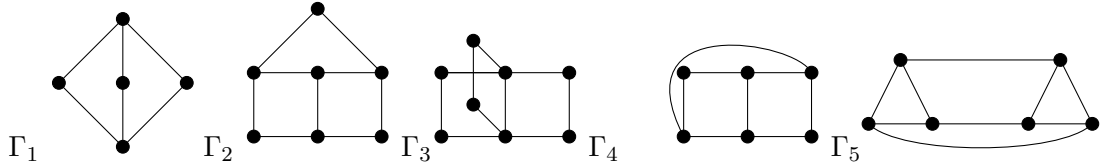
\begin{rem}\label{properties of graph}
 Let $G$ be a torsion-free group,  $C$ be a finite  subset of $G$ containing the identity element and $ B $ be  a $ k $-atom of $ C $, for some positive integer number $k$. Since $ xB $ is  a $ k $-atom of $ Cy $ for all $ x\in G $ and $y\in C^{-1}$,  we may consider the  graph $ P(xB,Cy) $  which by Lemma {\rm\ref{iso}} is isomorphic to the graph $ P(B,C) $.  
\end{rem}
 %------------------------------------------------------------------------------------------------------------------------
\vspace{1cc}
\section {\bf Small product sets}
\begin{lem}\label{4atom}
Let $G$ be a torsion-free group and  $C$ be a  finite  subset of $G$ containing the identity element such that $ \left\langle C\right\rangle  $ is not abelian. If $ |C|=3 $ and $ \kappa_4(C)=4 $, then $ \alpha_4(C)=4 $.
\end{lem}
\begin{proof}
 Let $ A $ be a $ 4 $-atom of $ C $. So, $ |AC|=|A|+4 $. Suppose, for a contradiction,  that $ |A|>4 $. It follows from Lemma \ref{lem4H} that $ |A|\geq 8 $ and $ r_{AC}(x)\geq 2 $, for all $ x\in AC $. If $ |A|=8 $, then $ |AC|=\frac{|A||C|}{2} $ that according to the proof of Proposition \ref{conz} and  \cite[Theorem 1.3]{PS}, is a contradiction. Hence, $ |A|\geq 9 $. By Lemma \ref{lem4H}, for every $ a\in A $, there are $ s_a,t_a,r_a\in C^{-1}\setminus \{1\} $ such that $ A_a=\{a,as_a,as_at_a,as_at_ar_a\}\subseteq A $. Since  $ \left\langle C\right\rangle  $  is a non-abelian group  and the group has no element of order 2 and 3 as it is torsion-free, it follows that  $ s_at_a\neq 1 $, $ t_ar_a\neq 1 $ and $ s_at_ar_a\neq 1 $. Hence, for every $ a\in A $, $ |A_a|=4 $. Since $ |C^{-1}\setminus \{1\}|=2 $, there are $ 8 $ choices for
the ordered triple $(s_a,t_a,r_a)$. So, as $ |A|\geq 9 $, there exist distinct elements $ a,b\in A $ such that $ A_a=A_b $. Therefore, $ ba^{-1}A_a=A_b\subseteq ba^{-1}A\cap A $, contradicting \cite[Lemma 1]{Hamidoune}. Hence, $ |A|=4 $. This completes the proof.
\end{proof}
\begin{lem}\label{7part}
Let $G$ be a torsion-free group and  $C$ be a  finite  subset of $G$ containing the identity element such that  $|C|=3$ and $ \left\langle C\right\rangle  $ is not abelian. If $ \kappa_4(C)=4 $ and $ B $ is  a $ 4$-atom of $ C $, then  there exists $ h\in C $ such that  $ Ch^{-1}=\{1,a,b\} $ and  exactly one of the following statements holds:
\begin{itemize}
\item[$ (1) $]$ a^3=b^2 $ and there exists $ a'\in G $ such that $ B=a'\{1,b,a,a^2\} $.
\item[$ (2) $]  $ aba^{-2}b=1 $ and there exists $ a'\in G $ such that $ B=a'\{1,a,b^{-1},b^{-1}a\} $.
\item[$ (3) $] $ ab^{-1}aba^{-1}b=1 $ and there exists $ a'\in G $ such that $ B=a'\{1,a,ba^{-1},aba^{-1}\} $.
\item[$ (4) $] $ ba^{2}b^{-1}a=1 $ and there exists $ a'\in G $ such that $ B=a'\{1,a,b^{-1},b^{-1}a^{-1}\} $.
\item[$ (5) $]$ b^{-1}a^{2}ba=1 $ and there exists $ a'\in G $ such that $ B=a'\{1,a^{-1},a^{-2},b\} $.
\item[$ (6) $] $ ba^{-2}b^{-1}a=1 $ and there exists $ a'\in G $ such that $ B=a'\{1,a,b^{-1},b^{-1}a\} $.
\item[$ (7) $] $ a^{-2}bab^{-1}=1 $ and there exists $ a'\in G $ such that $ B=a'\{1,a^{-1},a^{-2},ba^{-1}\} $.
\item[$ (8) $] $ a^{-2}b^{-2}=1 $ and there exists $ a'\in G $ such that $ B=a'\{1,a^{-1},a^{-2},b\} $.
\item[$ (9) $] $ a^{-2}b^2=1 $ and there exists $ a'\in G $ such that $ a'B$ is one of the sets $\{1,a^{-1},b^{-1},ab^{-1}\} $,\\ $\{1,a^{-1},b^{-1}a,ab^{-1}\} $, $\{1,a,b,a^{-1}\} $, $\{1,a^{-1},a,ab^{-1}\} $ and $ \{1,a,b^{-1},ba^{-1}\} $.
\end{itemize}
\end{lem}
\begin{proof}
It follows from Lemma \ref{4atom} that $ |B|=4 $. Let $ B=\{g_1,g_2,g_3,g_4\} $. As $ \kappa_4(C)=4 $, $ |BC|=8 $. In view of Definition \ref{defgraph}, $ P(B,C)$ must have at least $ |B||C|-|BC|=4 $ edges. Therefore, it follows from Lemma \ref{cayley} that $ P(B,C)$ contains  one of the graphs in Figure \ref{gamma}.

Let  $ C=\{1,x,y\} $. Suppose   that $ P(B,C)$ contains an square $\mathcal{C}$ as Figure \ref{gamma}. Suppose first that $\mathcal{C}$ is of type $ (ii) $. Hence, according to Remark \ref{squ}  $ r(\mathcal{C})=1 $ satisfies one of the relations marked by ``*"s ($ 24 $ cases) in Table \ref{C4}.  In the following we show that the cases $ (5),(7),(17),(21),(25) $ and $ (29) $ satisfy the part $ (1) $:
\begin{figure}
%\subfloat{$ H_9$}
\begin{tikzpicture}[scale=.9]
\draw [fill] (0,0) circle
[radius=0.1] node  [left]  {$ g_1 $};
\draw [fill] (1,0) circle
[radius=0.1] node  [right]  {$ g_2 $};
\draw [fill] (0,-1) circle
[radius=0.1] node  [left]  {$ g_4 $};
\draw [fill] (1,-1) circle
[radius=0.1] node  [right]  {$ g_3 $};
\draw   (0,0)  --  (0,-1) ;
\draw   (0,0)  --  (1,0) ;
\draw   (1,-1)  --  (1,0) ;
\draw   (1,-1)  --  (0,-1) ;
\end{tikzpicture}
\subfloat{$C_4$}
\qquad \qquad\qquad\qquad
\begin{tikzpicture}[scale=.9]
\draw [fill] (0,0) circle
[radius=0.1] node  [left]  {$ g_1 $};
\draw [fill] (1,0) circle
[radius=0.1] node  [right]  {$ g_4 $};
\draw [fill] (0,-1) circle
[radius=0.1] node  [left]  {$ g_3 $};
\draw [fill] (1,-1) circle
[radius=0.1] node  [right]  {$ g_2 $};
\draw   (0,0)  --  (0,-1) ;
\draw   (0,0)  --  (1,0) ;
\draw   (1,-1)  --  (0,0) ;
\draw   (1,-1)  --  (0,-1) ;
\end{tikzpicture}
\subfloat{$\Gamma$}
\caption{}\label{gamma}
\end{figure}
\begin{itemize}
\item[$(5)$] $ x^{-2}y^{-1}xy^{-1}=1 $: In this case $ [1,x,1,x,1,y,x,y]\in \mathcal{T}(\mathcal{C}) $ and therefore we may assume $ g_1=g_2x $, $ g_2=g_3x $, $ g_3=g_4y $ and $ g_4 x=g_1y $. In view of Remark \ref{properties of graph}, we may assume $ g_1=1 $ and we can replace $C$ with  $ Cx^{-1} $. Now, if we let $ a=x^{-1} $ and $ b=yx^{-1} $, then  $ Cx^{-1}=\{1,a,b\} $, $ a^3=b^2 $ and $ B=\{1,a,a^2,b\} $.
\item[$(7)$] $ x^{-2}y^2x^{-1}=1 $: In this case $ [1,x,1,x,y,1,y,x]\in \mathcal{T}(\mathcal{C}) $ and therefore we may assume $ g_1=g_2x $, $ g_2=g_3x $, $ g_3y=g_4 $ and $ g_4 y=g_1x $. Now, by assuming that  $ g_1=1 $, $ x^2B=\{1,x,x^2,y\} $. 
\item[$(17)$]$ x^{-1}(y^{-1}x)^2y^{-1}=1 $: In this case $ [1,x,1,y,x,y,x,y]\in \mathcal{T}(\mathcal{C}) $ and therefore we may assume $ g_1=g_2x $, $ g_2=g_3y $, $ g_3x=g_4y $ and $ g_4 x=g_1y $. According to Remark \ref{properties of graph}, we may assume $ g_1=1 $ and we can replace $C$ with  $ Cx^{-1} $. Now, if we let $ b=x^{-1} $ and $ a=yx^{-1} $, then we have $ Cx^{-1}=\{1,a,b\} $, $ a^3=b^2 $ and $ B=\{1,a,a^2,b\} $.
\item[$(21)$] $ x^{-1}y^3x^{-1}=1 $: In this case, since $ [1,x,y,1,y,1,y,x]\in \mathcal{T}(\mathcal{C}) $, by the same argument as above, it can be seen that $ xB=\{1,y,y^2,x\} $. 
 \item[$(25)$] $ x^{-1}yx^{-1}y^{-2}=1 $: By interchanging $ x $ and $ y $ in $(5)$ and with the same discussion, it can be seen that this case satisfies the part $ (1) $.
\item[$(29)$]$(x^{-1}y)^2x^{-1} y^{-1}=1 $: By interchanging $ x $ and $ y $ in $(17)$ and with the same argument, it can be seen that this case satisfies the part $ (1) $.
\end{itemize}
In the following we show that the cases $ (14) $ and $ (22) $ in Table \ref{C4} satisfy the part $ (2) $:
\begin{itemize}
\item[$(14)$] $ x^{-1}y^{-1}x^2y^{-1}=1 $: In this case $ [1,x,1,y,x,1,x,y]\in \mathcal{T}(\mathcal{C}) $ and therefore we may assume $ g_1=g_2x $, $ g_2=g_3y $, $ g_3x=g_4 $ and $ g_4 x=g_1y $. Now, by assuming that  $ g_1=1 $,  $ xB=\{1,x,y^{-1},y^{-1}x\} $.
\item[$(22)$] $ x^{-1}y^{2}x^{-1}y^{-1}=1 $: By interchanging $ x $ and $ y $ in $(14)$ and with the same argument, it can be seen that this case satisfies the part $ (2) $.
\end{itemize}
With the same argument as above, it can be seen that if $ r(\mathcal{C})=1 $ satisfies  the case $ (26) $ in Table \ref{C4}, then since $ [1,x,y,x,1,y,x,y]\in \mathcal{T}(\mathcal{C}) $, this case satisfies the part $ (3) $. In the following we show that the cases $ (11),\;(16) $ and $ (28) $ satisfy the part $ (4) $:
 \begin{itemize}
\item[$(11)$] $ x^{-1}y^{-1}x^{-1}yx^{-1}=1 $: In this case $ [1,x,1,y,1,x,y,x]\in \mathcal{T}(\mathcal{C}) $ and therefore we may assume $ g_1=g_2x $, $ g_2=g_3y $, $ g_3=g_4x $ and $ g_4y =g_1x $. Now, by assuming that  $ g_1=1 $,  $ xB=\{1,x,y^{-1}x^{-1},y^{-1}\} $.
\item[$(16)$] $ x^{-1}y^{-1}xy^{-2}=1 $: By interchanging $ x $ and $ y $ in $(11)$ and with the same discussion, it can be seen that this case satisfies the part $ (4) $.
\item[$(28)$] $ (x^{-1}y)^2yx^{-1}=1 $: In this case $ [1,x,y,x,y,1,y,x]\in \mathcal{T}(\mathcal{C}) $ and therefore we may assume $ g_1=g_2x $, $ g_2y=g_3x $, $ g_3=g_4y $ and $ g_4 y=g_1x $. Now, by assuming that   $ g_1=1 $, if we let $ b=x^{-1} $ and $ a=yx^{-1} $, then $ Cx^{-1}=\{1,a,b\} $, $ ba^2b^{-1}a=1 $ and $ b^{-1}B=\{1,b^{-1},a,b^{-1}a^{-1}\} $.
\end{itemize}
 In the following we show that the cases $ (8),\;(13) $ and $ (23) $ in Table \ref{C4} satisfy the part $ (5) $:
 \begin{itemize}
\item[$(8)$] $ x^{-2}yx^{-1}y^{-1}=1 $: In this case $ [1,x,1,x,y,x,1,y]\in \mathcal{T}(\mathcal{C}) $ and therefore we may assume $ g_1=g_2x $, $ g_2=g_3x $, $ g_3y=g_4x $ and $ g_4 =g_1y $. Now, by assuming that $ g_1=1$,  $ B=\{1,x^{-1},x^{-2},y\} $.
\item[$(13)$] $ x^{-1}y^{-2}xy^{-1}=1 $: By interchanging $ x $ and $ y $ in $(8)$ and with the same discussion, it can be seen that this case satisfies the part $ (5) $.
\item[$(23)$] $ x^{-1}y(yx^{-1})^2=1 $: In this case $ [1,x,y,1,y,x,y,x]\in \mathcal{T}(\mathcal{C}) $ and therefore we may assume $ g_1=g_2x $, $ g_2y=g_3 $, $ g_3y=g_4x $ and $ g_4 y=g_1x $.  Now, by assuming that  $ g_1=1 $, if we let $ b=x^{-1} $ and $ a=yx^{-1} $, then we have $ Cx^{-1}=\{1,a,b\} $, $ a^2bab^{-1}=1 $ and $ B=\{1,a^{-1},a^{-2},b\} $.
\end{itemize}
In the following we show that the cases $ (15),\;(18) $ and $ (27) $ in Table \ref{C4} satisfy the part $ (6) $:
 \begin{itemize}
\item[$(15)$] $ x^{-1}y^{-1}xyx^{-1}=1 $: In this case $ [1,x,1,y,x,1,y,x]\in \mathcal{T}(\mathcal{C}) $ and therefore we may assume $ g_1=g_2x $, $ g_2=g_3y $, $ g_3x=g_4 $ and $ g_4y =g_1x $.  Now, by assuming that $ g_1=1 $,   $ xB=\{1,x,y^{-1},y^{-1}x\} $.
\item[$(18)$] $ x^{-1}yxy^{-2}=1 $: By interchanging $ x $ and $ y $ in $(15)$ and with the same discussion, it can be seen that this case satisfies the part $ (6) $.
\item[$(27)$] $ (x^{-1}y)^2xy^{-1}=1 $: In this case $ [1,x,y,x,y,1,x,y]\in \mathcal{T}(\mathcal{C}) $ and therefore we may assume $ g_1=g_2x $, $ g_2y=g_3x $, $ g_3y=g_4 $ and $ g_4 x=g_1y $.  Now, by assuming that  $ g_1=1 $, if we let $ b=y^{-1} $ and $ a=xy^{-1} $, then  $ Cy^{-1}=\{1,a,b\} $, $ ba^{-2}b^{-1}a=1 $ and $ a^2b^{-1}B=\{1,a,b^{-1},b^{-1}a\} $.
\end{itemize}
In the following we show that the cases $ (6),\;(19) $ and $ (20) $ in Table \ref{C4} satisfy the part $ (7) $:
 \begin{itemize}
\item[$(6)$] $ x^{-2}yxy^{-1}=1 $: In this case $ [1,x,1,x,y,1,x,y]\in \mathcal{T}(\mathcal{C}) $ and therefore we may assume $ g_1=g_2x $, $ g_2=g_3x $, $ g_3y=g_4 $ and $ g_4x =g_1y $.  Now, by assuming that $ g_1=1 $,  $ B=\{1,x^{-1},x^{-2},yx^{-1}\} $.
\item[$(19)$] $ x^{-1}y(xy^{-1})^2=1 $: In this case $ [1,x,y,1,x,y,x,y]\in \mathcal{T}(\mathcal{C}) $ and therefore we may assume $ g_1=g_2x $, $ g_2y=g_3 $, $ g_3x=g_4y $ and $ g_4 x=g_1y $.  Now, by assuming that $ g_1=1 $, if we let $ b=y^{-1} $ and $ a=xy^{-1} $, then  $ Cy^{-1}=\{1,a,b\} $, $ a^2ba^{-1}b^{-1}=1 $ and $ B=\{1,a^{-1},a^{-2},ba^{-1}\} $.
\item[$(20)$] $ x^{-1}y^{2}xy^{-1}=1 $: By interchanging $ x $ and $ y $ in $(6)$ and with the same discussion, it can be seen that this case satisfies the part $ (7) $.
\end{itemize}
In the following we show that the cases $ (4),(9) $ and $ (33) $ in Table \ref{C4} satisfy the part $ (8) $:
\begin{itemize}
\item[$(4)$] $ x^{-2}y^{-2}=1 $: In this case $ [1,x,1,x,1,y,1,y]\in \mathcal{T}(\mathcal{C}) $ and therefore we may assume $ g_1=g_2x $, $ g_2=g_3x $, $ g_3=g_4y $ and $ g_4 =g_1y $.  Now, by assuming that $ g_1=1 $,  $ B=\{1,x^{-1},x^{-2},y\} $.
\item[$(9)$] $x^{-1}(x^{-1}y)^2x^{-1}=1$: In this case $ [1,x,1,x,y,x,y,x]\in \mathcal{T}(\mathcal{C}) $ and therefore we may assume $ g_1=g_2x $, $ g_2=g_3x $, $ g_3y=g_4x $ and $ g_4 y=g_1x $.  Now, by assuming that  $ g_1=1 $, if we let $ b=x^{-1} $ and $ a=yx^{-1} $, then  $ Cx^{-1}=\{1,a,b\} $, $ b^2a^{2}=1 $ and $ b^{-2}B=\{1,a^{-1},a^{-2},b\} $.
\item[$(33)$] $y^{-1}(y^{-1}x)^2y^{-1}=1$: By interchanging $ x $ and $ y $ in $(9)$ and with the same discussion, it can be seen that this case satisfies the part $ (8) $.
\end{itemize}
On the other hand, by  Lemma \ref{two cycle with one edge}, if the square $\mathcal{C}$ is of type $ (i) $, then the part (9) holds. Now, suppose that  $ P(B,C)$ contains the graph $ \Gamma $ as Figure \ref{gamma}.  If the triangle which is included in the graph $ \Gamma$ is of   type $ (i) $, then there exists $ x\in BC $ such that $ r_{BC}(x)=3 $ and therefore since $ |BC|=8 $, $ P(B,C)$ must have at least 5 edges and therefore  $ P(B,C)$ contains a graph isomorphic to the graph in Figure \ref{two cycle with one edge in common}. So, we may assume that the triangle which is included in the graph $ \Gamma$ is of   type $ (ii) $. Then  lemma \ref{triii} implies $ r(\mathcal{C})\in \{ x^{-1}y^{-1}xy^{-1},  x^{-1}yx^{-1}y^{-1},x^{-1}y^{2}x^{-1} \} $. Suppose that  $ T=[h_1,h'_1,h_2,h'_2,h_3,h'_3] $ is the 6-tuple to $\mathcal{C}$ corresponding to the arrangement $g_1,g_2, g_3$. Clearly, there exist $ h,h'\in C $ such that $ g_1h=g_4h' $. Note that we may assume  $ g_1=1 $ (by Remark \ref{properties of graph}) and  $ h\neq h_1 $ and $ h\neq h'_3 $ since otherwise  $ P(B,C)$ contains a graph isomorphic to the graph in Figure \ref{two cycle with one edge in common}. It is easy to see that every choice for $  r(\mathcal{C}) $ leads to holding the part $ (9) $. We will prove  the case $ r(\mathcal{C})= x^{-1}y^{-1}xy^{-1} $ and $ T=[1,x,1,y,x,y] $ only as the proof of the other cases are similar. In this case, $ g_1=g_2x $, $ g_2=g_3y $, $ g_3x=g_1y $ and $ g_1x=g_4h' $, where $ h'\in \{1,y\} $. Hence, $ B $ is one of the sets $ \{1,x,x^{-1},yx^{-1}\} $ and $ \{1,xy^{-1},x^{-1},yx^{-1}\} $. Now, if we let $ b=yx^{-1} $ and $ a= x^{-1}$, then $ Cx^{-1}=\{1,a,b\} $, $ a^2=b^2 $ and $ B  $ is one of the sets $ \{1,a,a^{-1},b\} $  and   $ \{1,b,b^{-1},a\} $. This completes the proof.
\end{proof}
\begin{lem}\label{5atom}
Let $G$ be a torsion-free group and  $C$ be a  finite  subset of $G$ containing the identity element such that  $|C|=3$ and $\left\langle C\right\rangle$ is not abelian. If  $ \kappa_5(C)=4 $, then $ \alpha_5(C)=5 $.
\end{lem}
\begin{proof}
If $\left\langle C\right\rangle$ is isomorphic to the Klein bottle group, then  $\left\langle C\right\rangle$ is a unique product group \cite{PI} and so there is nothing to prove. So, we may assume that $\left\langle C\right\rangle$ is not isomorphic to the Klein bottle group. Let $ A $ be  a 5-atom of $ C $. Suppose, for a contradiction, that $ |A|>5 $. Thus, Lemma \ref{lem4H} and Remark \ref{set} imply  $ |AC|\geq 8 $ and $ r_{AC}(t)\in \{ 2,3\}  $, for all $ t\in AC $. If $ |A|=8 $ and $ |A|=9 $, then $|AC|=\frac{3|A|}{2} $ and $|AC|=\frac{3|A|-1}{2} $, respectively, that according to the proof of Propositions \ref{conz} and \ref{conu},  \cite[Theorem 1.3]{PS} and  \cite[Proposition 4.12]{a55} are contradictions. Thus, $ |A|\geq 10 $. Now, suppose that $ |A|\geq 13 $. Let $C=\{1,x,y\} $. Observe that, since $ r_{AC}(ay)\geq 2 $, for all $ a\in A $, $ |AC|=|A\{1,x\}| $. Hence, $ |A\cap Ax|= |A|-4\geq 9 $. Therefore,  by Lemma \ref{lem4H}, for every $ a\in A\cap Ax $, there are $ s_a,t_a,r_a\in C^{-1}\setminus \{1\} $ such that $ A_a=\{a,ax^{-1},ax^{-1}s_a,ax^{-1}s_at_a,ax^{-1}s_at_ar_a\}\subseteq A $. Since $\left\langle C\right\rangle$ is a non-abelian torsion-free group and also $\left\langle C\right\rangle$ is not isomorphic to the Klein bottle group, for every $ a\in A\cap Ax $, $ |A_a|=5 $. Since $ |C^{-1}\setminus \{1\}|=2 $, there are $ 8 $ choices for
the ordered triple $(s_a,t_a,r_a)$. So, as $ |A\cap Ax|\geq 9 $, there exist distinct elements $ a,b\in A $ such that $ A_a=A_b $. Therefore, $ ba^{-1}A_a=A_b\subseteq (ba^{-1})A\cap A $, contradicting \cite[Lemma 1]{Hamidoune}. Hence, $ 10\leq |A|\leq 12$. 
Since for each $ t\in AC $,  $ r_{AC}(t)\in \{2,3\}$, if $ T=\{t\in AC\;|\; r_{AC}(t)=3\} $, then $ |T|=|A||C|-2|AC| $ and therefore $ |T|=|A|-8 $. It is clear that, for each element $ t\in T $, there exists a triangle  of type $ (i) $ in $ S(A,C) $.  Since $\left\langle C\right\rangle$ is not isomorphic to the Klein bottle group, it follows from Lemma \ref{triii} that $ S(A,C) $ contains no triangle  of type $ (ii) $. So, if $ |A|=n $, then $ S(A,C) $ contains exactly $ n-8 $ triangles. On the other hand, by Lemma \ref{cayley} and Lemma \ref{two cycle with one edge},  each two triangles in $ S(A,C) $ have at most one vertex in common. Also, since for each $ g\in A $ and $ h\in C $,  $ r_{AC}(gh)\in \{2,3\}$, it is easy to see that the degree of each vertex $ g $ in $ S(A,C) $ is equal to $ 3+s $, where $ s=|\{h\;|\;h\in C, r_{AC}(gh)=3\}| $. So, it is clear that the degree of every vertex in  $ S(A,C) $ is equal to 3, 4, 5 or 6.   Then, $ S(A,C) $ has the following properties: (1) contains exactly $ n-8 $ triangles each two of which triangles have at most one vertex in common; (2)  every vertex of it has degree $ 3,4,5 $ or $ 6 $; (3) the degree of each its vertex $ g $ is equal to $ 3+s' $, where $ s' $ is the number of triangles which $ g $ is a common vertex between them; (4) contains no subgraph isomorphic to one of the graphs in Figure {\rm \ref{forbiddensub}}. For each $ n\in \{10,11,12\} $, we check and see that there is no graph with $ n  $ vertices and with above properties. Thus, $ |A|\notin \{10,11,12\} $ and therefore $ \alpha_5(C)\neq 5 $. This completes the proof.
\end{proof}
\begin{thm}\label{maintheorem}
Let $G$ be a torsion-free group and  $C$ be a  finite  subset of $G$ containing the identity element such that  $|C|=3$ and $\left\langle C\right\rangle$ is not abelian. If $ \kappa_5(C)\leq 4 $  and $ B $ is  a $ 5$-atom of $ C $, then  there exist $ h\in C $ and  $ a'\in G $ such that  $ Ch^{-1}=\{1,a,b\} $,  $ a^{2}=b^2 $ and  $ a'B$ is one of the sets  $\{1,a^{-1},b^{-1},b^{-1}a,ab^{-1}\} $, $\{1,a,b,a^{-1},ab^{-1}\} $ and $\{1,a^{-1},a,ab^{-1},aba^{-1}\} $.
\end{thm}
\begin{proof}
 Let $ C=\{1,x,y\} $. Since $ \kappa_5(C)\leq 4 $, it follows from \cite[Lemma 8]{Hamidoune} that  $ \kappa_4(C)=\kappa_5(C)= 4 $. Let $ B $ be a $ 5 $-atom of $ C $. By Lemma \ref{5atom}, $ |B|=5 $. Also, since $ \kappa_5(C)= 4 $, $ |BC|=9 $. So, there exist $ b\in B $ and $ c\in C $ such that $ r_{BC}(bc)=1 $ since otherwise we must have $ |BC|\leq\frac{|B||C|}{2} $, a contradiction. By the choice of $ b $, if  $ B'=B\setminus \{b\} $, then clearly $ \kappa_4(C) \leq |\partial_C(B')|\leq |\partial_C(B)|=4 $. Therefore, $ B' $ is a $ 4$-atom of $ C $ and $ |\partial_C(B')|= 4$. So, one of the 9 parts of  Lemma \ref{7part} holds. Suppose first that the part (9) holds. In other word, we may assume  that $ x,y $ satisfy  $ x^{2}=y^{2} $ and   $ B $ is one of sets $\{1,x^{-1},y^{-1},xy^{-1},b\} $, $\{1,x^{-1},y^{-1}x,xy^{-1},b\} $, $\{1,x,y,x^{-1},b\} $, $\{1,x^{-1},x,xy^{-1},b\} $ and $ \{1,x,y^{-1},yx^{-1},b\} $. Since $ |B'C|=8 $ and $ |BC|=9 $,  there exists an element $a'\in  \{b,bx,by\} $ such that  $ BC=B'C\cup \{a'\} $ and $ T=\{b,bx,by\}\setminus \{a'\}\subseteq B'C $. We checked all possible choices for $ a' $, many of them lead to either  $\left\langle C\right\rangle$ has a non-trivial torsion element or   $\left\langle C\right\rangle$ is an abelian group, that are contradictions. Just in the following cases, we have no contradiction: if $ B=\{1,x^{-1},y^{-1},xy^{-1},b\} $, then $ B'C=\{1,x^{-1},y^{-1},xy^{-1},x,y^{-1}x,xy^{-1}x,y\} $ and  there are two possible  cases: (1) $ b=x^{-2} $ which leads $ \{bx,by\}=\{x^{-1},y^{-1}\}\subseteq B'C $, $ BC=B'C\cup \{x^{-2}\}$; (2) $ b=y^{-1}x $ which leads $ \{bx,b\}=\{y,y^{-1}x\}\subseteq B'C $, $ BC=B'C\cup \{y^{-1}xy\}$;
  if $ B=\{1,x^{-1},y^{-1}x,xy^{-1},b\} $, then $ B'C=\{1,x^{-1},y^{-1}x,xy^{-1},x,y^{-1}xy,xy^{-1}x,y\} $ and  the only possibility  is $ b=y^{-1} $ which leads to $ \{bx,by\}=\{1,y^{-1}x\}\subseteq B'C $, $ BC=B'C\cup \{y^{-1}\}$;  
 if $ B=\{1,x,y,x^{-1},b\} $, then $ B'C=\{1,x,y,x^{-1},xy,x^{-1}y,x^{2},yx\} $ and  the only possibility  is  $ b=x^{-1}y $ which leads to $ \{b,by\}=\{x^{-1}y,x\}\subseteq B'C $, $ BC=B'C\cup \{x^{-1}yx\}$; 
 if $ B=\{1,x^{-1},x,xy^{-1},b\}  $, then $ B'C=\{1,x^{-1},x,xy^{-1},x^2,xy^{-1}x,xy,y\} $ and  and  there are two possible cases: (1)  $ b=y $ which leads to $ \{b,by\}=\{y,x^2\}\subseteq B'C $, $ BC=B'C\cup \{yx\}$; (2) $ b=xyx^{-1} $ which leads to $ \{b,bx\}=\{xy^{-1}x,xy\}\subseteq B'C $, $ BC=B'C\cup \{xyx^{-1}y\}$; 
  if $ B=\{1,x,y^{-1},yx^{-1},b\} $, then $ B'C=\{1,y^{-1},x,yx^{-1},x^2,yx^{-1}y,xy,y\} $ and  the only possibility  is  $ b=y $ which leads to $ \{b,by\}=\{x^{2},y\}\subseteq B'C $, $ BC=B'C\cup \{yx\}$. 
By a same argument as above we checked all other parts (1) to (8) of Lemma  \ref{7part}, in every of  these parts, every choice for $ b $ implies that either  $\left\langle C\right\rangle$ has a non-trivial torsion element or   $\left\langle C\right\rangle$ is an abelian group, that are contradictions. This completes the proof.
\end{proof}
The following Corollary follows from Theorem \ref{maintheorem} and Remark \ref{klein}.
\begin{cor}\label{mainresut}
 Let $G$ be a torsion-free group and  $C$ be a  finite  subset of $G$ containing the identity element such that  $|C|=3$ and $\left\langle C\right\rangle$  is neither abelian nor isomorphic to the Klein bottle group. Then  for all subsets $ B\subseteq G $ with $ |B|\geq 5 $,
$
|BC|\geq |B|+5.
$
\end{cor}
\begin{lem}\label{maintheorem6}
Let $G$ be a torsion-free group and  $C$ be a  finite  subset of $G$ containing the identity element such that  $|C|=3$ and $\left\langle C\right\rangle$ is not abelian. If $ \kappa_6(C)\leq 4 $ and $ B $ is a $6$-atom of $ C $, then  there exist $ h\in C $ and $ a'\in G $ such that  $ Ch^{-1}=\{1,a,b\} $,  $ a^{2}=b^2 $ and  $ a'B=\{1,a^{-1},b^{-1},b^{-1}a,ab^{-1},a^{-2}\}$.
\end{lem}
\begin{proof}
 Let $ C=\{1,x,y\} $. Without loss of generality, we may assume that $1\in B$. Since $ \kappa_6(C)\leq 4 $, $ |BC|\leq |B|+4 $ and therefore by Corollary \ref{mainresut} we may assume that $\left\langle C\right\rangle$ is isomorphic to  the Klein bottle group. Hence, $\left\langle C\right\rangle$ is a unique product group \cite{PI}. Thus, $ |B|=6 $ and there exist $ b\in B $ and $ c\in C $ such that $ r_{BC}(bc)=1 $. By the choice of $ b $, if  $ B'=B\setminus \{b\} $, then by \cite[Lemma 8]{Hamidoune},  $4\leq \kappa_4(C) \leq |\partial_C(B')|\leq |\partial_C(B)|=4 $. Therefore, $ B' $ is a $ 5$-atom of $ C $ and $ |\partial_C(B')|= 4$. So, in view of   Lemma \ref{maintheorem} we may assume that $ x,y $ satisfy  $ x^{2}=y^{2} $ and   $ B $ is one of sets $\{1,x^{-1},y^{-1},y^{-1}x,xy^{-1},b\} $, $\{1,x,y,x^{-1},xy^{-1},b\} $ and $\{1,x^{-1},x,xy^{-1},xyx^{-1},b\} $. Since $ |B'C|=9 $ and $ |BC|=10 $,  there exists an element $a'\in  \{b,bx,by\} $ such that  $ BC=B'C\cup \{a'\} $ and $ T=\{b,bx,by\}\setminus \{a'\}\subseteq B'C $. We checked all possible choices for $ a' $, many of them lead to either $ G $ has a non-trivial torsion element or being $\left\langle C\right\rangle$ is an abelian group, that are contradictions. Just in the following cases, we have no contradiction: if $ B=\{1,x^{-1},y^{-1},y^{-1}x,xy^{-1},b\} $, then $ B'C=\{1,x^{-1},y^{-1},xy^{-1},x,y^{-1}x,xy^{-1}x,y,yx^{-1}y\} $ and  the only possibility is $ b=x^{-2} $ which leads to $ \{bx,by\}=\{x^{-1},y^{-1}\}\subseteq B'C $, $ BC=B'C\cup \{x^{-2}\}$; if $ B=\{1,x,y,x^{-1},xy^{-1},b\} $, then $ B'C=\{1,x^{-1},y,xy^{-1},x,x^{2},xy,yx,xy^{-1}x\} $ and  the only possibility is $ b=xy^{-1}x $ which leads to $ \{bx,b\}=\{xy,xy^{-1}x\}\subseteq B'C $, $ BC=B'C\cup \{xy^{-1}xy\}$; if $ B=\{1,x^{-1},x,xy^{-1},xyx^{-1},b\} $, then $ B'C=\{1,x^{-1},y,xy^{-1},x,x^{2},xy,xy^{-1}xy,xy^{-1}x\} $ and  the only possibility is $ b=y $ which leads to $ \{by,b\}=\{y,y^{2}\}\subseteq B'C $, $ BC=B'C\cup \{yx\}$. Hence, in every case, there is $ a\in G $ such that $ aB $ is equal to $\{1,x^{-1},y^{-1},y^{-1}x,xy^{-1},x^{-2}\} $. This completes the proof.
\end{proof}
\begin{cor}\label{mainresut7}
Let $G$ be a torsion-free group and  $C$ be a  finite  subset of $G$ containing the identity element such that  $|C|=3$ and $\left\langle C\right\rangle$ is not abelian. Then  for all subsets $ B\subseteq G $ with $ |B|\geq 7 $, $|BC|\geq |B|+5$.
\end{cor}
\begin{proof}
Let $ C=\{1,x,y\} $. By Corollary \ref{mainresut}, if $\left\langle C\right\rangle$ is not isomorphic to the Klein bottle group, then there is nothing to prove. Hence, we may assume that $\left\langle C\right\rangle$ is  isomorphic to the Klein bottle group. It is sufficient to prove $ \kappa_7(C)\geq 5 $. Suppose, for a contradiction, that  $ \kappa_7(C)\leq 4 $. Let $ A $ be a $7$-atom of $ C $. Since  the Klein bottle group is a unique product group \cite{PI}, $ |A|=7 $ and also there exist $ a\in A $ and $ c\in C $ such that $ r_{AC}(ac)=1 $. By the choice of $ a $, if  $ A'=A\setminus \{a\} $, then by \cite[Lemma 8]{Hamidoune},  $4\leq \kappa_6(C) \leq |\partial_C(A')|\leq |\partial_C(A)|\leq4 $. Therefore, $ |\partial_C(A')|=4 $. So, in view of   Lemma \ref{maintheorem6} we may assume that $ x,y $ satisfy  $ x^{2}=y^{2} $ and   $ A= \{1,x^{-1},y^{-1},y^{-1}x,xy^{-1},x^{-2},a\} $. Since $ |A'C|=10 $ and $ |AC|=11 $,  there exists an element $a'\in  \{a,ax,ay\} $ such that  $ AC=A'C\cup \{a'\} $ and $ \{a,ax,ay\}\setminus \{a'\}\subseteq A'C $. It is not hard to see that every choice for $ a' $ leads to  either $\left\langle C\right\rangle$ has a non-trivial torsion element or  $\left\langle C\right\rangle$ is an abelian group, that are contradictions. Hence,  $ \kappa_7(C)\geq 5 $. This completes the proof. 
\end{proof}
\begin{lem}\label{u.p1}
Let $G$ be a  unique product group and  $C$ be a  finite  subset of $G$ containing the identity element such that  $|C|=4$ and $\left\langle C\right\rangle$ is not abelian. Then $\kappa_7(C)\geq 6$.
\end{lem}
\begin{proof}
Suppose, for a contradiction,  that $\kappa_7(C)\leq 5.$ So,  \cite[Lemma 8]{Hamidoune} implies $\kappa_7(C)= 5$. Let $ A $ be a 7-atom of $ C $ containing the identity element. Since $ G $ is a unique product group, $ |A|=7 $. Suppose  that $\left\langle C\right\rangle \neq \left\langle  A\right\rangle  $. Hence, either $ C$ intersects at least two right cosets of $ \left\langle A\right\rangle $ or $A$ intersects at least two left cosets of $ \left\langle C\right\rangle $.  Suppose first that $ C$ intersects at least two right cosets of $ \left\langle A\right\rangle $. Let $C_1$ be one of these intersections. Thus, \ref{kemperman} implies
$12=|AC|=|AC_1|+|A(C\setminus C_1)|\geq 2|A|+|C|-2= 16$, a contradiction. Now suppose that $A$ intersects at least two left cosets of $ \left\langle C\right\rangle $. Let $A_1$ be one of these intersections. Thus, \ref{kemperman} implies
$12=|AC|=|A_1C|+|(A\setminus A_1)C|\geq |A|+2|C|-2= 13$, a contradiction. So, $ \left\langle A\right\rangle=\left\langle C\right\rangle $. \\
Since $ G $ is a unique product group, there exist $ a\in A $ and $ c\in C $ such that $ r_{AC}(ac)=1 $. Let $ C'=C\setminus \{c\} $. Suppose that  $c=1$. Then if we replace $C$ with $Ct^{-1}$, where $t\in C\setminus \{1\}$, then   $\left\langle Ct^{-1}\right\rangle$ is not abelian, $1\in Ct^{-1}$, $A$ is a  7-atom of $ Ct^{-1} $  and $ r_{AC}(at^{-1})=1 $. So, without loss of generality, we may assume that $ 1\in C' $. It is clear hat $ ac\notin Ac\cap AC' $ and therefore
\begin{equation}\label{2}
|AC|=|AC'|+|Ac|-|AC'\cap Ac|\geq |AC'|+1.
\end{equation}
Thus, $ |AC'|\leq 11 $ and therefore Corollary \ref{mainresut7} implies  $ \left\langle C'\right\rangle \neq \left\langle C\right\rangle$.   Since $ \left\langle A\right\rangle = \left\langle C\right\rangle $, $ A$ intersects at least two left cosets of $ \left\langle C'\right\rangle $. 
Partition $ A=A_1\cup A_2\cup\cdots \cup A_t $, where each $ A_i $ is the nonempty intersection
of $ A $ with some left coset of   $ \left\langle C'\right\rangle $. Then $ |AC'|=\sum_{i=1}^{t}|A_iC'| $.  By \ref{kemperman},   $ |A_iC'|\geq |A_i|+2 $, for all $ i\in \{1,\ldots,t\} $, and by the main result of \cite{BF}, if $ |A_iC'|= |A_i|+2 $, then $ A_i $ and $ C' $ are left and right progressions with common ratio, respectively. According to the above condition,  the only possibility is  $ t=2 $ and $ |A_iC'|= |A_i|+2 $ for each $ i\in\{1,2\} $. Therefore, without loss of generality, we may assume that $ A=\{1,x,\ldots,x^{(i-1)}\}\cup \{r,rx,\ldots,rx^{(j-1)}\} $ and $ C'=\{1,x,x^2\} $, where $ r,x\in G\setminus \{1\} $, $ i+j=7 $ and $ i>j $.  Since $ \left\langle C\right\rangle= \left\langle A\right\rangle $ is not abelian, $ rx\neq xr $ and $ xc\neq cx $. Note that we may assume that for each $ a'\in A $, $r_{AC}(a')\geq 2  $ since otherwise by the choice $ c=1 $, $ \left\langle C'\right\rangle $ is not abelian which leads to a contradiction.
It is clear that $ (i,j)\in \{(6,1),(5,2),(4,3)\} $. Suppose first that $ (i,j)=(6,1) $. In this case  $ AC'=\{1,x,\ldots,x^7,r,rx,rx^2\}  $. Since $ |AC|=12 $, there exists $ h\in Ac $ such that $ AC=AC'\cup \{h\} $ and $ Ac\setminus \{h\}\subseteq AC' $. So, since $ \left\langle C\right\rangle $ is not abelian, there exists $ B\subseteq \{c,xc,x^2c,x^3c,x^4c,x^5c\} $ such that $ |B|\geq 4 $ and $ B\subseteq \{r,rx,rx^2\} $. Thus, there are distinct elements  $ i',j'\in \{0,1,2,3,4,5\} $ such that $ x^{i'}c=x^{j'}c $ which leads to $ G $ has a non-trivial torsion-element, a contradiction. Now, suppose that  $ (i,j)=(5,2)$. In this case, $ AC'=\{1,x,\ldots,x^6,r,rx,rx^2,rx^3\}  $ and there exists $ h\in Ac $ such that $ AC=AC'\cup \{h\} $ and so $ Ac\setminus \{h\}\subseteq AC' $. Since $\left\langle C\right\rangle$ is not abelian, there exists $ B\subseteq \{c,xc,x^2c,x^3c,x^4c\}$ such that $ |B|\geq 4 $ and  $B \subseteq \{r,rx,rx^2,rx^3\} $. Hence, we must have $ |B|=4 $ since otherwise $ G $ has a non-trivial torsion element, a contradiction. Thus, $ h\in \{c,xc,x^2c,x^3c,x^4c\} $.  Therefore, $ \{rc,rxc\}\subseteq AC' $ and therefore $ \{rc,rxc\}\subseteq \{1,x,\ldots,x^6\} $. Thus, there are distinct elements  $ i',j'\in \{0,1,\ldots,6\} $ such that $ rc=x^{i'} $ and $ rxc=x^{j'} $. On the other hand, since $ r_{AC}(r)\geq 2 $,  $ G $  is a torsion-free group and $ \left\langle A\right\rangle $ is not abelian, there exists $ s\in  \{0,1,2,3,4\} $ such that $ r=x^sc $. Then, we have $ c^2=x^{(i'-s)} $, $ cxc=x^{(j'-s)} $ and $ c^{-1}xc=x^{(j'-i')} $. It is clear that $ s\notin\{i',j'\} $ since otherwise we have contradiction with $\left\langle C\right\rangle$ is a non-abelian torsion-free group. So, $ c^{-1}x^{(i'-s)}c=x^{(j'-i')(i'-s)} $ which leads to $ x^{(j'-i'-1)(i'-s)}=1 $. Hence, we must have $ j'-i'=1 $  implies $ xc=cx $ that is a contradiction. For the last case, suppose that  $ (i,j)=(4,3)$. In this case, $ AC'=\{1,x,\ldots,x^5,r,rx,rx^2,rx^3,rx^4\}  $ and there exists $ h\in Ac $ such that $ AC=AC'\cup \{h\} $ and therefore $ Ac\setminus \{h\}\subseteq AC' $. Consider two cases: (1) $ h=c $: in this case $ \{rc,rxc,rx^2c\}\subseteq AC' $ and by the same argument as the latter case, we get into a contradiction; (2) $ h\neq c $: so, $ c\in AC' $ and since $\left\langle C\right\rangle$ is not abelian, $ c\in \{r,rx,rx^2,rx^3,rx^4\} $. Thus, there exists $ i'\in \{0,1,2,3,4\} $ such that $ c=rx^{i'} $.  On the other hand, since $ r_{AC}(1) $ and $ r_{AC}(r) $ are greater than or equal to 2,  $\left\langle C\right\rangle$  is a non-abelian torsion-free group and $ \left\langle A\right\rangle $ is not abelian, there exists $ s\in  \{0,1,2,3\} $ and $ s'\in  \{0,1,2\} $ such that $ r=x^sc $ and $ 1= rx^{s'}c$. Then we have $ c^2=x^{(i'-s')} $, $ x^scx^{s'}c=1 $ and $ cx^{i'}c^{-1}=x^{-s} $. It is clear that $ i'\neq s' $ since otherwise $ G $ has a non-trivial torsion element, a contradiction.  Hence, $ c x^{i'(i'-s')}c^{-1}=x^{-s(i'-s')} $ which leads to $ x^{(-s-i')(i'-s')}=1 $. Then, $ s=i'=0 $ and therefore $ r=c $ and $ r^2=x^{-s'} $. So, the only possibility  is $ s'=2 $ which implies that   $ \{r^2,xr\}\subseteq Ac $ and $ \{r^2,xr\}\nsubseteq AC' $ since otherwise $ G $ has a non-trivial torsion element. Hence, $ |AC|\geq 13, $ a contradiction. Thus, $\kappa_7(C)\geq 6$ and this completes the proof.
\end{proof}
\begin{thm}
Let $G$ be a  unique product group and  $C$ be a  finite  subset of $G$ containing the identity element such that  $\left\langle C\right\rangle$ is not abelian.  Then for all subsets $ B $ of $ G $ with $ |B|\geq 7 $,
$|BC|\geq |B|+|C|+2$.
\end{thm}
\begin{proof}
It is sufficient to prove $ \kappa_7(C)\geq |C|+2 $. Suppose the contrary and choose a counter-example with minimal $|C|$. Since  $\left\langle   C\right\rangle   $  is not abelian and by Lemma \ref{u.p1} and Corollary \ref{mainresut7}, $ |C|\geq 5 $. Let $A$ be a 7-atom of $ C $ containing the identity element. Since $ G $ is a unique product group,  $ |A|=7 $.  By the same argument as the proof of Lemma \ref{u.p1}, we have $\left\langle C\right\rangle= \left\langle  A\right\rangle  $.
Since $ G $ is a unique product group, there exists an element $ x\in AC $ that can be represented in
a unique way in the form $ac$ with $ a\in A $ and $ c\in C $. Let $ C'=C\setminus \{c\} $. Then $ x\notin Ac\cap AC' $ and therefore
\begin{equation}\label{1}
|AC|=|AC'|+|Ac|-|AC'\cap Ac|\geq |AC'|+1.
\end{equation}
Without loss of generality, we may assume that $ 1\in C' $. Consider two cases: (1) $ \left\langle C\right\rangle= \left\langle  C'\right\rangle  $:  then the minimality of $ |C| $ implies that $ |AC'|\geq |A|+|C'|+2 $ and by \ref{1},  $ |AC|\geq |A|+|C|+2 $ that is a contradiction; (2) $ \left\langle C\right\rangle\neq \left\langle  C'\right\rangle  $:  then $  A$ intersects at least two left cosets of $ \left\langle C'\right\rangle $. Let $A_1$ be one of these
intersections.  Hence, \ref{kemperman} implies
$
|AC'|=|A_1C'|+|(A\setminus A_1)C'|\geq |A|+ 2|C'|-2\geq |A|+|C'|+2,
$ 
and by \ref{1},  $ |AC|\geq |A|+|C|+2 $ that is a contradiction.  This completes the proof.
\end{proof}
\begin{cor}
Let $  \alpha$ and $  \beta$ be non-zero elements of $ \mathbb{F}[G] $, the group algebra of any torsion-free group over an
arbitrary field. If $|supp(\alpha)| = 3$ and $\alpha \beta= 0$, then $|supp(\beta)| \geq 12$.
\end{cor}
\begin{proof}
 By \cite[Theorem 1.4]{AT}, it is sufficient to prove $|supp(\beta)| \notin \{10,11\}$. Since $\alpha\beta=0$, $\beta^*\alpha^*=0$, where  $supp(\alpha^*)=supp(\alpha)^{-1}$ and $supp(\beta^*)=supp(\beta)^{-1}$. Let $ B=supp(\beta^*) $ and $ C=supp(\alpha^*) $.  By \cite[Lemma 2.7]{AJ} and since $\beta^*\alpha^* h^{-1} = 0$, for all $  h\in C$, we may assume that $ 1\in C $ and  $ G=\left\langle C\right\rangle  $. Since $\beta^*\alpha^* = 0$, $ r_{BC}(x)\geq 2 $, for all $ x\in BC $. Hence, $ |BC|\leq \frac{3|B|}{2} $. On the other hand, since $ G $ is not abelian (see \cite[Theorem 26.2]{PI}, Corollary \ref{mainresut7} implies $ |BC|\geq |B|+5 $. Thus, $ |B|+5\leq |BC|\leq \frac{3|B|}{2}$. Suppose first that $ |B|=10 $. Then $ |BC|=15 $ and therefore  $ r_{BC}(x)= 2 $, for all $ x\in BC $. Hence, if we let $ \delta=\sum_{b\in B}b $ and $ \gamma=\sum_{c\in C}c $, then $ \delta,\gamma\in \mathbb{F}_2[G] $ and $ \delta\gamma=0 $ which by \cite[Theorem 1.3]{PS}  is a contradiction. Now, suppose that $ |B|=11 $. Hence, $ |BC|=16 $ and therefore there exist $ x\in BC $ such that $ r_{BC}(x)=3 $ and $ r_{BC}(x')=2 $, for all $ x'\in BC\setminus \{x\} $. Suppose that $ (b,c)\in R_{BC}(x) $. Hence, if we let $ \delta=\sum_{b\in B}b $ and $ \gamma=\sum_{c\in C}c $, then $ \delta,\gamma\in \mathbb{F}_2[G] $ and $ b^{-1}\delta\gamma c^{-1}=1 $ which by \cite[Proposition 4.12]{a55}  is a contradiction.
\end{proof}
\begin{cor}
Let $ \delta$ and $  \gamma$ be  elements of the group algebra of any torsion-free group over an
arbitrary field. If $|supp(\delta)| = 3$ and $\delta\gamma= 1$, then $|supp(\gamma)| \geq 10$.
\end{cor}
\begin{proof}
By \cite[Theorem 1.7]{AT}, it is sufficient to prove $|B| \neq 9$. Suppose, for a contradiction,  that $ |B|=9 $. Let $ B=supp(\gamma) $ and $ C=supp(\delta) $.   Without loss of generality, we may assume that $\gamma\delta = 1$. Since $\gamma\delta = 1$, there exist $ b\in B $ and $ c\in C $ such that $ bc=1 $ and also  $ r_{BC}(x)\geq 2 $, for all $ x\in BC\setminus \{1\} $. Therefore, $ |BC|\leq 14 $. Also,  since $b^{-1}\gamma\delta c^{-1} =1$,   by \cite[Lemma 2.9]{AJ}  we may assume that $ 1\in  C $ and  $ G=\left\langle C\right\rangle  $. Now   since $ G $ is not abelian (see \cite[Theorem 26.2]{PI},   Corollary \ref{mainresut7} implies $ |BC|\geq 14$.
 Then,  $ |BC|= 14 $. Thus, the only possibility is  $ r_{BC}(1)=1 $ and  $ r_{BC}(x)= 2 $, for all $ x\in  BC\setminus \{1\} $. So, if we let $ \alpha=\sum_{b\in B}b $ and $ \beta=\sum_{c\in C}c $, then $ \alpha,\beta\in \mathbb{F}_2[G] $ and $ \alpha\beta=1 $ which by \cite[Proposition 4.12]{a55}   is a contradiction. This completes the proof.
\end{proof}
%-------------------------------------------------------------------------------------------------

\small\noindent{\bf Alireza Abdollahi} \\
Department of Mathematics, University of Isfahan, Isfahan 81746-73441,
 Iran,\\ and \\ School of Mathematics, Institute for Research in Fundamental Sciences (IPM), P.O.Box 19395-5746, Tehran, Iran\\
E-mail address: {\tt a.abdollahi@math.ui.ac.ir}\\

\noindent {\bf Fatemeh Jafari} \\
 Department of Mathematics, University of Isfahan, Isfahan 81746-73441, Iran.\\
E-mail address: {\tt f$\_ $jafari@sci.ui.ac.ir}
\end{document}